\documentclass[production,12pt]{style}
\usepackage{latexsym}\usepackage{amsmath}\usepackage{xspace}\usepackage{enumerate}\usepackage{amsfonts}\usepackage{amscd}\usepackage{syntonly}\usepackage{amssymb}

\newtheorem{thm}{Theorem}

\newtheorem{lem}[thm]{Lemma}
\newtheorem{prop}[thm]{Proposition}

\newtheorem{claim}{Claim}

\newtheorem{rem}[thm]{Remark}
\newtheorem{question}{Question}

\newenvironment{pf}[1][Proof]{\noindent \textbf{#1:} }{\hspace{\stretch{1}}}\newenvironment{enui}{\begin{enumerate}[(i)]}{\end{enumerate}}

\newcommand{\BAR}[1]{{\overline{#1}}}

\newcommand\dd{\partial}

\newcommand\al{{\alpha}}

\newcommand\eps{\varepsilon}
\newcommand\Om{\Omega}
\newcommand\om{\omega}

\renewcommand\phi{\varphi}

\newcommand\K{\operatorname{K}}

\newcommand\kkk{\operatorname{k}}

\newcommand{\N}{\mathbb{N}}
\newcommand{\Z}{\mathbb{Z}}

\newcommand{\R}{\mathbb{R}}

\newcommand\Cross{{\times}}
\newcommand{\wt}[1]{\widetilde{#1}}
\newcommand\iso{\cong}

\newcommand\nn{\nonumber}

\newcommand\C{\mathbb C}

\newcommand\U{\operatorname{U}}

\newcommand\x{\times}
\newcommand\wo{\setminus}
\newcommand\sub{\subseteq}
\newcommand\one{\mathbf{1}}
\newcommand\Ham{\operatorname{Ham}}

\newcommand\diam{\operatorname{diam}}
\newcommand\Diam{\operatorname{Diam}}
\newcommand\D{\mathbb{D}}

\newcommand\id{\operatorname{id}}
\newcommand\pr{\operatorname{pr}}

\newcommand\const{\equiv}

\newcommand\g{\mathfrak{g}}
\newcommand\lan{\langle}
\newcommand\ran{\rangle}

\newcommand\lf{{\operatorname{L}}}
\newcommand\kf{{\operatorname{K}}}
\newcommand\nf{{\operatorname{N}}}

\newcommand\then{{\Longrightarrow}}
\newcommand\cont{{\supseteq}}
\newcommand\HH{\mathcal{H}}
\newcommand\DD{\mathcal{D}}

\newcommand\Int{\operatorname{Int}}

\title{Hofer Geometry of a Subset of a Symplectic Manifold}

\author{Jan Swoboda and Fabian Ziltener}

\begin{document}

\begin{abstract} To every closed subset $X$ of a symplectic manifold $(M,\om)$ we associate a natural group of Hamiltonian diffeomorphisms $\Ham(X,\om)$. We equip this group with a semi-norm $\Vert\cdot\Vert^{X,\om}$, generalizing the Hofer norm. We discuss $\Ham(X,\om)$ and $\Vert\cdot\Vert^{X,\om}$ if $X$ is a symplectic or isotropic submanifold. The main result involves the relative Hofer diameter of $X$ in $M$. Its first part states that for the unit sphere in $\R^{2n}$ this diameter is bounded below by $\frac\pi2$, if $n\geq2$. Its second part states that for $n\geq2$ and $d\geq n+1$ there exists a compact set in $\R^{2n}$ of Hausdorff dimension at most $d$, with relative Hofer diameter bounded below by $\pi/\kkk(n,d)$, where $\kkk(n,d)$ is an explicitly defined integer.
\end{abstract}

\maketitle

\tableofcontents

\section{Motivation and main results}\label{sec:mot main}
The theme of this article is the following.\\

\begin{question}\label{q:small} 
How much symplectic geometry can a small subset of a symplectic manifold carry?
\end{question}
To be specific, we interpret ``small'' as ``of Hausdorff dimension bounded above by a given number''. In the article \cite{SZSmall} we gave some answers to this question in terms of the displacement energy of the subset, non-squeezing, and exoticness of symplectic structures. Here we look at this question from a dynamical point of view. The goal is to lay the foundations of a Hofer geometry for subsets of a symplectic manifolds, both from an absolute and relative view-point, and to explore this geometry in examples. 

\subsubsection*{Absolute Hofer geometry} 
Let $(M,\om)$ be a symplectic manifold and $X\sub M$ a closed subset. (For simplicity all manifolds in this paper are assumed to have empty boundary.) We define the set of \emph{Hamiltonian diffeomorphisms of $X$, $\Ham(X,M,\om)$}, as follows. 

Let $V:[0,1]\x M\to TM$ be a smooth time-dependent vector field on $M$. For every $t\in[0,1]$ we denote by $\phi_V^t$ the time-$t$ flow of $V$. Its domain is by definition the set $\DD_V^t$ of all points $x_0\in M$ for which the problem 
\[\dot x=V\circ x,\quad x(0)=x_0\] 
has a solution $x\in C^\infty([0,t],M)$. We say that $V$ is \emph{$X$-compatible} iff $X\sub\DD_V^1$, and $\phi_V^t(X)=X$, for every $t\in[0,1]$. For a function $H\in C^\infty([0,1]\x M,\R)$ we denote by $X_H:=X^\om_H$ its time-dependent Hamiltonian vector field, and we abbreviate $\phi_H^t:=\phi_{H,\om}^t:=\phi_{X_H}^t$. We define 
\begin{equation}\label{eq:HH M om X}\HH(M,\om,X):=\big\{H\in C^\infty([0,1]\x M,\R)\,\big|\,X_H\textrm{ is }X\textrm{-compatible}\big\},\end{equation}
\begin{equation}\label{eq:Ham X om}\Ham(X,\om):=\Ham(X,M,\om):=\big\{\phi_H^1|_X\,\big|\,H\in\HH(M,\om,X)\big\}.
\end{equation}
Note that for $X=M$
\[\HH(M,\om):=\HH(M,\om,M)\]
is the set of all functions $H\in C^\infty([0,1]\x M,\R)$ whose Hamiltonian time-$t$ flow is well-defined on $M$ and a diffeomorphism of $M$, for every $t\in[0,1]$. Furthermore, $\Ham(M,\om)$ is the set of all time-one flows of functions in $\HH(M,\om)$. The following result shows that $\Ham(X,\om)$ together with composition is a group, and that it naturally generalizes $\Ham(M,\om)$.  
\begin{prop}[Hamiltonian diffeomorphisms of a subset]\label{prop:Ham X om} The following statements hold.
\begin{enui}
\item\label{prop:Ham X om:group} The set $\Ham(X,\om)$ is a subgroup of the group of homeomorphisms of $X$. 
\item\label{prop:Ham X om:sympl}If $X$ is a symplectic submanifold of $M$ then 
\begin{equation}\label{eq:Ham X om |}\Ham(X,\om)=\Ham(X,\om|_X)\end{equation}
(where on the right-hand side we regard $X$ as a subset of itself).
\end{enui}
\end{prop}
The trickiest part of the proof of this result is the inclusion ``$\cont$'' in (\ref{eq:Ham X om |}). The idea is to extend a given Hamiltonian function $H:[0,1]\x X\to\R$ to a function $\wt H:[0,1]\x M\to\R$ in such a way that the restriction of the time-$t$ flow of $\wt H$ to $X$ agrees with the time-$t$ flow of $H$ (see Proposition \ref{prop:H wt H} below). 

We define the \emph{Hofer semi-norm on $\Ham(X,\om)$} to be the map 
\[\Vert\cdot\Vert^{X,\om}:\Ham(X,\om)\to[0,\infty]\] 
given as follows. Let $H\in C^\infty([0,1]\x M,\R)$. We define the \emph{Hofer norm of $H$ on $X$} to be 
\begin{equation}\label{eq:Vert H X}\Vert H\Vert_X:=\int_0^1\big(\sup_XH(t,\cdot)-\inf_XH(t,\cdot)\big)\,dt\in[0,\infty].\end{equation}
(It follows from Lemma \ref{le:f} below that this integral is well-defined.) For every $\phi\in\Ham(X,\om)$ we define  
\begin{equation}\label{eq:Vert X om}\Vert\phi\Vert^{X,\om}:=\inf\big\{\Vert H\Vert_X\,\big|\,H\in\HH(M,\om,X):\,\phi_H^1|_X=\phi\big\}.
\end{equation}
By the next result the map $\Vert\cdot\Vert^{X,\om}$ is a semi-norm, which naturally generalizes $\Vert\cdot\Vert^{M,\om}$. Furthermore, $\Vert\cdot\Vert^{M,\om}$ is a norm. We will use the following definition. Let $G$ be a group. By a \emph{semi-norm}\label{semi-norm} on $G$ we mean a map $\Vert\cdot\Vert:G\to[0,\infty]$ such that
\begin{eqnarray}
\label{eq:one}&\Vert\one\Vert=0,&\\
\label{eq:g -1}&\Vert g^{-1}\Vert=\Vert g\Vert,&\\
\label{eq:gh}&\Vert gh\Vert\leq\Vert g\Vert+\Vert h\Vert,&
\end{eqnarray}
for every $g,h\in G$. We call $\Vert\cdot\Vert$ a \emph{norm} iff also
\begin{equation}\label{eq:non-deg}\Vert g\Vert=0\then g=\one.
\end{equation}
We call $\Vert\cdot\Vert$ \emph{invariant} iff 
\begin{equation}\label{eq:inv}\Vert ghg^{-1}\Vert=\Vert h\Vert,\quad\forall g,h\in G.
\end{equation}
\begin{prop}[Hofer semi-norm for a subset]\label{prop:Vert} The following statements hold.
\begin{enui}
\item\label{prop:Vert:semi} The map $\Vert\cdot\Vert^{X,\om}$ is an invariant semi-norm.
\item\label{prop:Vert:sympl} Assume that $X$ is a symplectic submanifold of $M$. Then the map $\Vert\cdot\Vert^{X,\om}$ is a norm and
\begin{equation}\label{eq:Vert X om|}\Vert\cdot\Vert^{X,\om}=\Vert\cdot\Vert^{X,\om|_X}.
\end{equation}
\end{enui}
\end{prop}
The proof of this result is similar to the proof of Proposition \ref{prop:Ham X om}. 

For a general closed subset $X\sub M$ the map $\Vert\cdot\Vert^{X,\om}$ may be degenerate, i.e., not satisfy (\ref{eq:non-deg}). It is maximally degenerate, if $X$ is a connected isotropic submanifold. This is a consequence of the following result.
\begin{prop}\label{prop:X iso} If $X$ is a connected isotropic submanifold then 
\begin{equation}\label{eq:Vert X om 0}\Vert\cdot\Vert_X\const0:\HH(M,\om,X)\to[0,\infty].\end{equation}
\end{prop}

\subsubsection*{Relative Hofer geometry}
Let $Y\sub M$ be a closed subset containing $X$. We may compare the Hofer geometries of the sets $X$ and $Y$ as follows: We define the \emph{Hofer semi-norm on $X$ relative to $Y$} to be the map
\begin{eqnarray}\label{eq:Vert X Y om}&\Vert\cdot\Vert_X^{Y,\om}:\Ham(X,\om)\to[0,\infty],&\\
\label{eq:Vert phi X Y om}&\Vert\phi\Vert_X^{Y,\om}:=\inf\big\{\Vert\psi\Vert^{Y,\om}\,\big|\,\psi\in\Ham(Y,\om):\,\psi|_X=\phi\big\}.&
\end{eqnarray}
Intuitively, this map measures how short a Hamiltonian path on $X$ can be made inside $Y$. The definition (\ref{eq:Vert X Y om}) has the following natural properties. 
\begin{prop}[Relative Hofer semi-norm]\label{prop:Vert X Y Y'} The map $\Vert\cdot\Vert_X^{Y,\om}$ is a semi-norm. Furthermore, let $Y'\sub M$ be a closed subset such that $Y$ is contained in the interior of $Y'$. If $Y$ is compact and non-empty, then we have
\begin{equation}\label{eq:Vert X Y Y'}\Vert\cdot\Vert_X^{Y,\om}\geq\Vert\cdot\Vert_X^{Y',\om}.
\end{equation}
\end{prop}
In the case $X=Y$ we have, by definition,
\[\Vert\cdot\Vert_X^{X,\om}=\Vert\cdot\Vert^{X,\om}.\] 
However, in general, the semi-norms $\Vert\cdot\Vert_X^{Y,\om}$ and $\Vert\cdot\Vert^{X,\om}$ may differ a lot. As an example, a forth-coming article \cite[Corollary 7]{ZiHofer} contains the following result. 
\begin{thm}[Relative Hofer diameter]\label{thm:diam M M'} Let $(M,\om)$ and $(M',\om')$ be connected symplectic manifolds and $X'\sub M'$ a finite subset. Assume that $M$ is closed and $M'$ has positive dimension. Then we have 
\begin{equation}\Vert\cdot\Vert_{M\x X'}^{M\x M',\om\oplus\om'}\const0.
\end{equation}
\end{thm}
In contrast with this result, under the hypotheses of Theorem \ref{thm:diam M M'}, the absolute semi-norm $\Vert\cdot\Vert^{M\x X',\om\oplus\om'}$ is non-degenerate. This follows from Proposition \ref{prop:Vert}(\ref{prop:Vert:sympl}).

The relative Hofer semi-norm gives rise to the \emph{Hofer diameter of $X$ relative to $Y$}, which we define as
\begin{equation}\label{eq:diam X Y om}\diam(X,Y,\om):=\sup\big\{\Vert\phi\Vert_X^{Y,\om}\,\big|\,\phi\in\Ham(X,\om)\big\}.
\end{equation}
This quantity measures how much Hamiltonian dynamics of $Y$ is captured by the subset $X$. Our main result is motivated by the following instances of Question \ref{q:small}. 
\begin{question}[Hofer diameter of a subset]\label{q:subset} What is the relative Hofer diameter $\diam(X,M,\om)$ for a given (small) closed subset $X\sub M$?
\end{question}
We now fix a subset $X_0\sub M$ and a number $d\in[0,\infty)$. 
\begin{question}[Maximal Hofer diameter]\label{q:diam M om X d} What is the supremum of the numbers $\diam(X,M,\om)$, where $X$ is a compact subset of $X_0$, of Hausdorff dimension at most $d$? 
\end{question}
In order to state our result, we define the map 
\begin{equation}\label{eq:kkk}\kkk:\N\x[0,\infty)\to\N\cup\{\infty\}\end{equation}
as follows. For $(n,d)\in\N\x[0,\infty)$ we define $\kkk(n,d)$ to be the infimum of all sums $\sum_{i=1}^\ell k_i$, where $\ell\in\N$ is such that
\begin{equation}\label{eq:ell 2}\ell\geq2,
\end{equation}
and $k_1,\ldots,k_\ell\in\N$ are integers for which there exist numbers $n_i\in\N$, for $i=1,\ldots,\ell$, such that the following conditions hold:
\begin{eqnarray}\label{eq:n i k i}&n_i\geq k_i,&\\
\label{eq:k i n i d geq}&\sum_ik_in_i\geq n,\quad \sum_ik_i(2n_i-k_i)\leq d.&\\
\label{eq:2 min}&2\min\{n_1,\ldots,n_\ell\}\leq n.&
\end{eqnarray}
Our main result provides lower bounds on the quantities in Questions \ref{q:subset} and \ref{q:diam M om X d} in the case $(M,\om):=(\R^{2n},\om_0)$, with $X$ the unit sphere $S^{2n-1}$ (for Question \ref{q:subset}) and $X_0$ the closed unit ball $\BAR B^{2n}\sub\R^{2n}$:
\begin{thm}[Relative Hofer diameter of a small set]\label{thm:exists X} The following statements hold.
\begin{enui}
\item\label{thm:exists X:i}\label{thm:exists X:S} 
For every integer $n\in\{2,3,\ldots\}$ we have 
\begin{equation}\label{eq:diam S 2n - 1}\diam(S^{2n-1},\R^{2n},\om_0)\geq\frac\pi2.\end{equation}
\item\label{thm:exists X:ii}\label{thm:exists X:k'} For every integer $n\in\{2,3,\ldots\}$ and real number $d\in[n,2n-1]$ there exists a compact subset $X\sub\BAR B^{2n}$ of Hausdorff dimension at most $d+1$, such that 
\begin{equation}\label{eq:diam X}\diam\big(X,\R^{2n},\om_0\big)\geq\frac\pi{\kkk(n,d)}.
\end{equation}
\end{enui}
\end{thm}
The estimate (\ref{eq:diam S 2n - 1}) is sharp up to a factor of $16$. This follows from the argument after Proposition \ref{prop:Diam diam} below. The proof of Theorem \ref{thm:exists X} is based on a coisotropic intersection result proved by the authors in \cite{SZSmall}. As another key ingredient, given a pair $(X_0,\al)$, where $X_0\sub M$ is a subset and $\al\in\Om^1(M)$, we will define what it means for $(X_0,\al)$ to be ``rigidifying''. Given a compact subset $X_0$, we will prove a lower bound on the Hofer norm of a certain Hamiltonian diffeomorphism, if there exists a function $f:M\to\R$ for which $(X_0,df)$ is rigidifying and some other conditions hold (Lemma \ref{le:d X limsup} below). We show that these conditions are satisfied if there exists a certain Hamiltonian Lie group action (Lemma \ref{le:rigid}).

The next result summarizes some properties of the map $\kkk$, which occurs in part (\ref{thm:exists X:k'}) of Theorem \ref{thm:exists X}. We define the function $\K:\N\to\N$ by 
\begin{equation}\K(n):=\inf\big\{\sum_{i=1}^\ell k_i\,\big|\,\ell\in\N,\,k_1,\ldots,k_\ell\in\N:\,n=\sum_ik_i^2\big\}.\label{eq:k n inf}
\end{equation}
The first few values of this function are 
\[\begin{array}{rrrrrrrrrrrrrrrrr%%rrr
}
    n=1&2&3&4&5&6&7&8&9&10&11&12&13&14&15&16&17\\ %% &18&19&20\\
\K(n)=1&2&3&2&3&4&5&4&3& 4& 5& 6& 5& 6& 7& 4& 5 %%& 6 & 7& 6
\end{array}
\]
\begin{prop}\label{prop:kkk} For every $n\in\{2,3,\ldots\}$ we have
\begin{equation}\label{eq:kkk n d}\kkk(n,d)\leq2n-d,\quad\forall d\in[n,2n-2].\end{equation}
\begin{eqnarray}\label{eq:k K}&\kkk(n,n)=\K(n),\quad\textrm{if }n\neq k^2,\,\forall k\in\N,&\\
\label{eq:K}&\K(n)<\sqrt n+2^{\frac32}\sqrt[4]n.&
\end{eqnarray}
\end{prop}
This proposition implies explicit lower bounds on the right-hand side of inequality (\ref{eq:diam X}). 

To put Theorem \ref{thm:exists X} into perspective, for each open subset $U\sub M$, we define the \emph{extension relative Hofer diameter of $U$} to be
\[\Diam(U,M,\om):=\]
\[\sup\big\{\Vert \phi_H^1\Vert^{M,\om}\,\big|\,H\in\HH(M,\om):\textrm{ support of }H\sub [0,1]\x U\big\}\in[0,\infty]\]
(where $\HH(M,\om):=\HH(M,\om,M)$). This diameter measures the sizes of trivial extensions of Hamiltonian diffeomorphisms generated by functions with support in $[0,1]\x U$. Note that in contrast with this, the definition of $\diam(X,M,\om)$ involves the restriction of a map $\psi:M\to M$ to $X$. The two diameters are related to each other as follows:
\begin{prop}[Relative Hofer diameters]\label{prop:Diam diam} Let $(M,\om)$ be a symplectic manifold, $U\sub M$ an open subset, and $X\sub U$ a compact subset. Then
\begin{equation}\label{eq:Diam diam} \Diam(U,M,\om)\geq\diam(X,M,\om).
\end{equation}
\end{prop}
We denote by $B^{2n}(a)\sub\R^{2n}$ the open ball of radius $\sqrt{a/\pi}$ around 0. It follows from \cite[Corollary 2]{ZiHofer} and a cutoff argument that 
\[\Diam(B^{2n}(a),\R^{2n},\om_0)\leq8a.\] 
(The proof of this result is a variant of an argument by J.-C.~Sikorav.) Combining this with (\ref{eq:Diam diam}), it follows that  
\[\diam(S^{2n-1},\R^{2n},\om_0)\leq8\pi.\]
This shows that the estimate (\ref{eq:diam S 2n - 1}) in Theorem \ref{thm:exists X} is sharp up to a factor of $16$. 
\subsubsection*{Remarks}\label{subsubsec:remarks}
\begin{itemize}
\item{\bf On Theorem \ref{thm:exists X}.} A straight-forward calculation shows that 
\[\diam\big(\R^{2n},\om_0,rX\big)=r^2\diam\big(\R^{2n},\om_0,X\big),\] 
for every $X\sub\R^{2n}$ and $r\in\R$. Hence Theorem \ref{thm:exists X} implies ``rescaled versions'' of itself, e.g., that $\diam\big(S^{2n-1}(a),\R^{2n},\om_0\big)\geq\frac a2$ for every $n\geq2$ and $a>0$. Here $S^{2n-1}(a)\sub\R^{2n}$ denotes the sphere of radius $\sqrt{a/\pi}$ around 0. 

The number $\kkk(n,d)$ occuring in this result is a modified version of a quantity defined in \cite{SZSmall}. 

\item{\bf On further research.} In the subsequent article \cite{ZiHofer} Theorem \ref{thm:diam M M'} will be proved. 

In Section \ref{sec:coiso Hofer} below we will develop a framework for finding lower bounds on relative Hofer diameters. This technique can be exploited in further examples. 

\item{\bf On compact supports.} Analogously to the group $\Ham(X,\om)$ (as defined in (\ref{eq:Ham X om})), one can define the group $\Ham_c(X,\om)$ of Hamiltonian diffeomorphisms on $X$ generated by a compactly supported function. To see that $\Ham(X,\om)$ can be strictly larger than $\Ham_c(X,\om)$, consider the example 
\[M=X:=\R^{2n},\quad\om:=\om_0,\quad\phi^t:\R^{2n}=\C^n\to\C^n,\phi^t(x):=e^{it}x,\] 
for some $t\in\R\in2\pi\Z$. Then $\phi^t\in\Ham(\R^{2n},\om_0)\wo\Ham_c(\R^{2n},\om_0)$. 

For the purpose of this article it seems more natural to consider the group $\Ham(X,\om)$. One argument for this is that $\Ham(M,\om)\wo\Ham_c(M,\om)$ may contain physically relevant maps, as in the above example. (Here $\phi^t$ is the time-$t$ evolution of the harmonic oscillator.)

Furthermore, if we define the displacement energy $e(X,M,\om)$ of a subset $X\sub M$ based on $\Ham(M,\om)$ (see (\ref{eq:e X M om}) below), then there exist triples $(M,\om,X)$, for which $\BAR{X}$ is non-compact and $e(M,\om,X)<\infty$. (Take for example $(M,\om):=(\R^2,\om_0)$ and $X:=\R\x\{0\}$.) In contrast with this, if we base the definition of $e(X,M,\om)$ on $\Ham_c(M,\om)$ instead, then we have to take special care of subsets $X\sub M$ for which $\BAR{X}$ is non-compact.

Note also that unlike $\Ham_c$, $\Ham$ has the nice product property 
\[\phi\x\id\in\Ham\big(M\x M',\om\oplus\om'\big),\quad\forall\phi\in\Ham(M,\om),\]
for arbitrary symplectic manifolds $(M,\om)$ and $(M',\om')$. This gives rise to an estimate for the displacement energy of a product set. 
\end{itemize}
\subsubsection*{Related work}\label{subsubsec:related}
J.-C.~Sikorav proved that for every open subset $U\sub\R^{2n}$ the diameter $\Diam(U,\R^{2n},\om_0)$ is bounded above by $16$ times the proper displacement energy of $U$. (See \cite{Si} or Theorem 10, Section 5.6 in the book \cite{HZ}.)

On the other hand, let $(M,\om)$ be a \emph{closed} symplectic manifold with $\pi_2(M)=0$ and $U\sub M$ a non-empty open subset. Then it follows from the proof of Theorem 1.1.~in the paper \cite{Os} by Y.~Ostrover that $\Diam(U,M,\om)=\infty$. 

The \emph{absolute} Hofer diameter 
\[\diam(M,\om):=\diam(M,\om,M)=\Diam(M,\om,M)\]
has been calculated for many closed symplectic manifolds. In all known examples it is infinite. For a recent overview and references, see the article by D.~McDuff \cite{McD}. 

In \cite{SZSmall} we considered Question \ref{q:small} from a different point of view, obtaining a stable displacement-energy-Gromov-width inequality, non-squeezing results, and existence of a stably exotic structure on $\R^{2n}$. These results are consequences of the key result, Theorem \ref{thm:N phi N} below. They involve functions similar to $\kkk$ (as defined in (\ref{eq:kkk})).
\subsubsection*{Organization of the article}\label{subsubsec:org}
In Section \ref{subsec:proofs:prop:Ham X om,Vert} we start by proving the first parts of Propositions \ref{prop:Ham X om} and \ref{prop:Vert} in a parallel way. Then we do the same for the second parts. In Section \ref{subsec:proof:prop:X iso,Vert X Y Y',kkk,prop:Diam diam} we prove Propositions \ref{prop:X iso}, \ref{prop:Vert X Y Y'}, \ref{prop:kkk}, and \ref{prop:Diam diam}. 

In Section \ref{sec:coiso Hofer} we develop a framework for proving a lower bound on the relative Hofer diameter of some subset, and we prove Theorem \ref{thm:exists X}. In Subsection \ref{subsec:coiso} we state the key result about coisotropic intersections (Theorem \ref{thm:N phi N}), which we proved in the article \cite{SZSmall}. In Subsection \ref{subsec:rigid} we introduce some ``rigidifying property'' and show how this implies a lower bound on the relative Hofer norm of a certain Hamiltonian diffeomorphism (Lemma \ref{le:d X limsup}). We also prove a sufficient criterion for the ``rigidifying property'' (Lemma \ref{le:rigid}). In Subsection \ref{subsec:proof:thm:exists X} we prove Theorem \ref{thm:exists X}.

The appendix contains some auxiliary results about symplectic geometry, point-set topology, and manifolds, which are used in the proofs of the results of Section \ref{sec:mot main}. 
\subsubsection*{Acknowledgements}
A considerable part of the work on this project was done during the second author's stay at the Max Planck Institute for Mathematics, Bonn. He would like to express his gratitude to the MPIM for the invitation and the generous fellowship. 
\section{Proofs of the propositions}\label{sec:proofs}
\subsection{Proofs of Propositions \ref{prop:Ham X om} and \ref{prop:Vert}}\label{subsec:proofs:prop:Ham X om,Vert}

We start by proving the first parts of Propositions \ref{prop:Ham X om} and \ref{prop:Vert} in a parallel way, post-poning the proofs of the second parts to page \pageref{proof:prop:Ham X om:sympl}. %%%  and \pageref{proof:prop:Vert:sympl}. %%% keep, since page might change

We need the following. Let $M$ be a $C^\infty$-manifold and $U\sub[0,1]\x M$ an open subset. We denote by $\pi:TM\to M$ the canonical projection. Let $V:U\to TM$ be a smooth map such that $\pi\circ V(t,x)=x$, for every $(t,x)\in U$. (If $U=[0,1]\x M$ then this means that $V$ is a time-dependent vector field on $M$.) We denote $V^t:=V(t,\cdot)$, for $t\in[0,1]$. We define $\DD_V$ to be the set of all pairs $(t_0,x_0)\in [0,1]\x M$ for which there exists a solution $x\in C^\infty([0,t_0],M)$ of the equations
\begin{equation}\label{eq:x V}x(0)=x_0,\quad (t,x(t))\in U,\quad\dot x(t)=V^t\circ x(t),\quad\forall t\in[0,1].\end{equation}
Furthermore, we define the flow of $V$ to be the map 
\[\DD_V\ni(t_0,x_0)\mapsto \phi_V^{t_0}(x_0):=\phi_V(t_0,x_0):=x(1)\in M,\]
where $x\in C^\infty([0,1],M)$ the unique solution of (\ref{eq:x V}). 

In the following, $(M,\om)$ is a symplectic manifold and $X\sub M$ a closed subset. Let $H,K\in C^\infty([0,1]\x M,\R)$. We define 
\begin{equation}\label{eq:BAR H}\BAR H:\DD_{X_H}\to\R,\quad\BAR H^t:=-H^t\circ\phi_H^t,\end{equation}
\begin{equation}\label{eq:H K}H\# K:\big\{(t,x)\in[0,1]\x M\,\big|\,x\in\phi_H^t(\DD_{X_H}^t)\big\}\to\R,\end{equation}
\[(H\# K)^t:=H^t+K^t\circ(\phi_H^t)^{-1}.\]
It follows from Remark \ref{rem:phi X t} below that the inverse $(\phi_H^t)^{-1}$ exists and hence $H\# K$ is well-defined, and that the domains of the functions $\BAR{H}$ and $H\# K$ are open subsets of $[0,1]\x M$. Their Hamiltonian vector fields are defined on the same sets. 

Let $X\sub M$ be a closed subset. 
\begin{lem}\label{le:BAR H K} If $H,K\in\HH(M,\om,X)$ then we have
\begin{eqnarray}\label{eq:X DD}&X\sub\DD_{X_{\BAR H}}^1,&\\
\label{eq:X DD H K}&X\sub\DD_{X_{H\# K}}^1,&\\
\label{eq:phi BAR H t}&\phi_{\BAR H}^t|_X=\phi_H^t|_X^{-1},\quad\forall t\in[0,1],&\\
\label{eq:phi H K t}&\phi_{H\# K}^t|_X=\phi_H^t\circ\phi_K^t|_X,\quad\forall t\in[0,1].&
\end{eqnarray}
\end{lem}
\begin{proof}[Proof of Lemma \ref{le:BAR H K}] These assertions follow from arguments as in the proof of \cite[Chapter 5, Proposition 1]{HZ}.
\end{proof}
\begin{proof}[Proof of Proposition \ref{prop:Ham X om}(\ref{prop:Ham X om:group})]\setcounter{claim}{0} Let $\phi\in\Ham(X,\om)$. We show that $\phi$ is a bijection on $X$: We choose $H\in\HH(M,\om,X)$ such that $\phi_H^1|_X=\phi$. By Remark \ref{rem:phi X t} below the map $\phi_H^1$ is injective. Furthermore, by the definition of $\HH(M,\om,X)$, we have $\phi(X)=\phi_H^1(X)=X$. It follows that $\phi$ is a bijection from $X$ to itself. 
\begin{claim}\label{claim:phi -1} We have $\phi^{-1}\in\Ham(X,\om)$. 
\end{claim}
In the proof of this claim we will denote by $\Int A$ the interior of a subset $A\sub M$. 
\begin{proof}[Proof of Claim \ref{claim:phi -1}] We define $\BAR H$ as in (\ref{eq:BAR H}). By (\ref{eq:X DD}) (Lemma \ref{le:BAR H K}) we have $X\sub\DD_{X_{\BAR H}}^1$. Since $X$ is closed and $\DD_{X_{\BAR H}}^1$ is open, it follows that there exist closed sets $A_0,A_1$ such that $X\sub\Int A_1$, $M\wo\DD_{X_{\BAR H}}^1\sub\Int A_0$ and $A_0\cap A_1=\emptyset$. By Lemma \ref{le:Urysohn} below there exists a function $f\in C^\infty(M,\R)$ such that $f\const i$ on $A_i$, for $i=0,1$. We define $\wt H:[0,1]\x M\to\R$ by $\wt H^t(x):=f(x)\BAR H^t$, if $x\in A_0$, and $\wt H^t(x):=0$, otherwise. 

Note that $\wt H^t=\BAR H^t$ on $A_1$. Using that $X\sub\Int A_1$, it follows that $X\sub\DD_{X_{\wt H}}^t$ and $\phi_{\wt H}^t|_X=\phi_{\BAR H}^t|_X$, for every $t\in[0,1]$. Combining this with the equality (\ref{eq:phi BAR H t}) of Lemma \ref{le:BAR H K}, it follows that $\phi^{-1}=\phi_{\wt H}^1|_X$. Condition (\ref{eq:phi BAR H t}) implies that $\wt H\in\HH(M,\om,X)$. Hence it follows that $\phi^{-1}\in\Ham(X,\om)$. This proves Claim \ref{claim:phi -1}.
\end{proof}

A similar argument, using (\ref{eq:X DD H K},\ref{eq:phi H K t}) in Lemma \ref{le:BAR H K} shows that $\Ham(M,\om)$ is closed under composition. The statement of Proposition \ref{prop:Ham X om}(\ref{prop:Ham X om:group}) is a consequence of this, Claim \ref{claim:phi -1} and the fact $\id_X\in\Ham(X,\om)$.
\end{proof}
\begin{proof}[Proof of Proposition \ref{prop:Vert}(\ref{prop:Vert:semi})]\setcounter{claim}{0} That the map $\Vert\cdot\Vert^{X,\om}$ is a semi-norm follows from an argument similar to the proof of Proposition \ref{prop:Ham X om}(\ref{prop:Ham X om:group}), using Lemma \ref{le:BAR H K}. Invariance follows from a straight-forward argument. 
\end{proof}
We continue by proving the second parts of Propositions \ref{prop:Ham X om} and \ref{prop:Vert} in a parallel way. We need the following two results. 
\begin{lem}\label{le:H|X} Assume that $X$ is a symplectic submanifold of $M$, and $H\in\HH(M,\om,X)$. Then we have 
\begin{equation}\label{eq:DD X H}\DD_{X^{\om|_X}_{H|_X}}^1=X,\quad\phi_{H|_X,\om|_X}^t=\phi_{H,\om}^t|_X,\,\forall t\in[0,1].\end{equation}
\end{lem}
For the proof of Lemma \ref{le:H|X} we need the following result, which will also be used for the proof of Proposition \ref{prop:X iso}.
\begin{lem}\label{le:T X} Assume that $X\sub M$ is a submanifold. Then for every $H\in\HH(M,\om,X)$, $t\in[0,1]$, and $x\in X$, we have
\begin{equation}\label{eq:X H t} X^\om_{H^t}(x)\in T_xX.
\end{equation}
\end{lem}

\begin{proof}[Proof of Lemma \ref{le:T X}]\setcounter{claim}{0} Let $x_0\in X$. For $t\in[0,1]$ we denote $x^t:=\phi_{H,\om}^t(x_0)$. By definition, we have $\phi_{H,\om}^t(X)=X$ and hence $x^t\in X$, for every $t\in[0,1]$. Let $t\in[0,1]$. It follows that
\[X_{H^t}(x^t)=\frac d{dt}x^t\in T_{x^t}X.\]
Hence, using again $\phi_{H,\om}^t(X)=X$, condition (\ref{eq:X H t}) follows. This proves Lemma \ref{le:T X}.
\end{proof}
\begin{proof}[Proof of Lemma \ref{le:H|X}]\setcounter{claim}{0}\label{proof:le:H|X} Let $x\in X$. Then for every $v\in T_xX$, we have 
\[\om|_X(X^{\om|_X}_{H^t|_X}(x),v)=d(H^t|_X)v=dH^t\,v=\om(X^\om_{H^t}(x),v).\]
Applying Lemma \ref{le:T X} and using the fact that $T_xX$ is a symplectic subspace of $T_xM$, it follows that 
\[X^{\om|_X}_{H^t|_X}(x)=X^\om_{H^t}(x).\]
The statement of Lemma \ref{le:H|X} follows.
\end{proof}
\begin{prop}\label{prop:H wt H} If $X$ is a symplectic submanifold then for every $H\in\HH(X,\om|_X)$ there exists $\wt H\in\HH(M,\om,X)$ such that 
\begin{eqnarray}\label{eq:phi wt H om t}&\phi_{\wt H,\om}^t|_X=\phi_{H,\om|_X}^t,\quad\forall t\in[0,1],&\\
\label{eq:wt H |}&\wt H|_{[0,1]\x X}=H.
\end{eqnarray} 
\end{prop}
In the proof of this result we will use the following notation. Let $(V,\om)$ be a symplectic vector space and $W\sub V$ a linear subspace. We denote
\begin{equation}\label{eq:W om}W^\om:=\big\{v\in V\,\big|\,\om(v,w)=0,\,\forall w\in W\big\}.\end{equation}
\begin{proof}[Proof of Proposition \ref{prop:H wt H}]\setcounter{claim}{0}\label{proof:prop:H wt H} Let $H\in\HH(X,\om|_X)$. By Proposition \ref{prop:tubular} below (applied with $N:=X$) there exists an embedding $\psi:E:=TX^\om:=\bigcup_{x\in X}T_xX^\om\to M$ satisfying the conditions (\ref{eq:psi N},\ref{eq:d psi x v}). We define $U:=\psi(E)$. This is an open subset of $M$ containing $X$. 

Since $M\wo U$ and $X$ are closed and do not intersect, there exists a pair of closed subsets $A_0,A_1\sub M$ such that $M\wo U\sub \Int(A_0)$, $X\sub\Int(A_1)$, and $A_0\cap A_1=\emptyset$. We choose a function $f$ as in Lemma \ref{le:Urysohn} below. We denote by $\pi:TX^\om\to X$ the canonical projection, and define 
\[\wt H:[0,1]\x M\to\R,\quad \wt H(t,x):=\left\{\begin{array}{ll}
f(x)H\big(t,\pi\circ\psi^{-1}(x)\big),&\textrm{if }x\in\psi(E),\\
0,&\textrm{otherwise.}
\end{array}\right.\]
This function is smooth and satisfies equality (\ref{eq:wt H |}). We define 
\[r:=\pi\circ\psi^{-1}:\Int A_1\to X.\] 
It follows from (\ref{eq:psi N},\ref{eq:d psi x v}) and our choice $E=TX^\om$ that this is a smooth retraction onto $X$, satisfying $\ker dr(x)=T_xX^\om$, for every $x\in X$. Let $t\in[0,1]$. Then we have $H^t\circ r=\wt H^t$ on $\Int A_1$. Hence Lemma \ref{le:X H r} below implies that $X^\om_{\wt H^t}(x)=X^{\om|_X}_{H^t}(x)$, for every $x\in X$. It follows that $\wt H\in\HH(M,\om,X)$ and equality (\ref{eq:phi wt H om t}) holds. This proves Proposition \ref{prop:H wt H}.
\end{proof}
We are now ready for the proofs of the remaining parts of Propositions \ref{prop:Ham X om} and \ref{prop:Vert}.

\begin{proof}[Proof of Proposition \ref{prop:Ham X om}(\ref{prop:Ham X om:sympl})]\setcounter{claim}{0}\label{proof:prop:Ham X om:sympl} We show the inclusion  ``$\sub$'' in {\bf (\ref{eq:Ham X om |})}: Let $\phi\in\Ham(X,\om)$. Choosing $H\in\HH(M,\om,X)$ such that $\phi_H^1|_X=\phi$, the inclusion  ``$\sub$'' is a consequence of Lemma \ref{le:H|X}.

The inclusion ``$\cont$'' in (\ref{eq:Ham X om |}) is a consequence of Proposition \ref{prop:H wt H}. This completes the proof of Proposition \ref{prop:Ham X om}(\ref{prop:Ham X om:sympl}).
\end{proof}

\begin{proof}[Proof of Proposition \ref{prop:Vert}(\ref{prop:Vert:sympl})]\setcounter{claim}{0}\label{proof:prop:Vert:sympl} We show the inequality ``$\geq$'' in {\bf (\ref{eq:Vert X om|})}: Let $\phi\in\Ham(X,\om)$. Let $H\in\HH(M,\om,X)$ be such that $\phi_H^1|_X=\phi$. By Lemma \ref{le:H|X} the conditions (\ref{eq:DD X H}) hold. By the definition of $\Vert\cdot\Vert^{X,\om|_X}$, it follows that
\[\Vert\phi\Vert^{X,\om|_X}\leq\Vert H|_X\Vert_X=\Vert H\Vert_X.\] 
It follows that $\Vert\phi\Vert^{X,\om|_X}\leq\Vert\phi\Vert^{X,\om}$. This proves inequality ``$\geq$'' in (\ref{eq:Vert X om|}).

The inequality ``$\leq$'' in {\bf (\ref{eq:Vert X om|})} is a consequence of Proposition \ref{prop:H wt H}. 

It remains to show that $\Vert\cdot\Vert^{X,\om}$ is non-degenerate, i.e., {\bf condition (\ref{eq:non-deg})} holds. By (\ref{eq:Vert X om|}) it suffices to prove the following claim.
\begin{claim}\label{claim:non-deg} If $X=M$ then condition (\ref{eq:non-deg}) holds. 
\end{claim}
For the proof of this claim we denote by $w(U):=w(U,\om|_U)$ the Gromov width of an open subset $U\sub M$. 

\begin{pf}[Proof of Claim \ref{claim:non-deg}] Let $\id\neq\phi\in\Ham(M,\om)$. Let $H\in\HH(M,\om,M)$ be such that $\phi_H^1=\phi$. We choose $x_0\in M$ such that $\phi(x_0)\neq x_0$ and an open neighborhood $U$ of $x_0$ with compact closure, such that $\phi(\BAR{U})\cap\BAR{U}=\emptyset$. 

Let $\eps>0$. By Lemma \ref{le:phi psi} below there exists $\psi\in\Ham_c(M,\om)$ such that $\psi|_{\BAR{U}}=\phi|_{\BAR{U}}$ and condition (\ref{eq:psi K d c}) holds. By a result by D.~McDuff and F.~Lalonde \cite[Theorem 1.1]{LMGeo} we have
\[\Vert\psi\Vert^{M,\om}_c\geq \frac12w(U).\]
Combining this with (\ref{eq:psi K d c}), and using that $\eps>0$ is arbitrary, it follows that $\Vert\phi\Vert^{M,\om}\geq \frac12w(U)>0$. This proves Claim \ref{claim:non-deg} and completes the proof of Proposition \ref{prop:Vert}(\ref{prop:Vert:sympl}).
\end{pf}\end{proof}
\subsection{Proofs of Propositions \ref{prop:X iso}, \ref{prop:Vert X Y Y'}, \ref{prop:kkk}, and \ref{prop:Diam diam}} \label{subsec:proof:prop:X iso,Vert X Y Y',kkk,prop:Diam diam}
\begin{proof}[Proof of Proposition \ref{prop:X iso}]\setcounter{claim}{0} Let $H\in\HH(M,\om,X)$. It suffices to show that for every path $x\in C^\infty([0,1],X)$ and $t\in[0,1]$ we have
\begin{equation}\label{eq:H t x 0} H^t\circ x(0)=H^t\circ x(1).
\end{equation}
To see this, we fix such a pair $(x,t)$. We have, for every $s\in[0,1]$, 
\begin{equation}\label{eq:d ds}\frac d{ds}(H^t\circ x)(s)=dH^t\dot x(s)=\om\big(X_{H^t}\circ x(s),\dot x(s)\big).\end{equation}
By Lemma \ref{le:T X} we have $X_{H^t}\circ x(s)\in T_{x(s)}X$, for every $s\in[0,1]$. Since $X$ is isotropic, it follows that the last expression in (\ref{eq:d ds}) vanishes. Hence (\ref{eq:d ds}) implies (\ref{eq:H t x 0}). This proves Proposition \ref{prop:X iso}.
\end{proof}
\begin{proof}[Proof of Proposition \ref{prop:Vert X Y Y'}]\setcounter{claim}{0} We prove the {\bf first statement}. Conditions (\ref{eq:one},\ref{eq:g -1}) (with $\Vert\cdot\Vert:=\Vert\cdot\Vert_X^{Y,\om}$) follow from straight-forward arguments. To see that condition (\ref{eq:gh}) holds, let $\phi_1,\phi_2\in\Ham(X,\om)$. Without loss of generality assume that $\Vert\phi_i\Vert_X^{Y,\om}<\infty$, for $i=1,2$. Let $\eps>0$. By definition of $\Vert\phi_i\Vert_X^{Y,\om}$ there exist maps $\psi_1,\psi_2\in\Ham(Y,\om)$ such that 
\begin{equation}\label{eq:psi i X}\psi_i|_X=\phi_i,\quad\quad\Vert\psi_i\Vert^{Y,\om}<\Vert\phi_i\Vert_X^{Y,\om}+\eps,\end{equation}
for $i=1,2$. We have $\psi:=\psi_1\circ\psi_2\in\Ham(Y,\om)$ and $\psi|_X=\phi_1\circ\phi_2$. It follows that
\begin{eqnarray}\nn\Vert\phi_1\circ\phi_2\Vert_X^{Y,\om}&\leq&\Vert\psi\Vert^{Y,\om}\\
\nn&\leq&\Vert\psi_1\Vert^{Y,\om}+\Vert\psi_2\Vert^{Y,\om}\\
\nn&<&\Vert\phi_1\Vert_X^{Y,\om}+\Vert\phi_2\Vert_X^{Y,\om}+2\eps,
\end{eqnarray}
where in the last inequality we used (\ref{eq:psi i X}). Since $\eps>0$ is arbitrary, the triangle inequality 
\[\Vert\phi_1\circ\phi_2\Vert_X^{Y,\om}\leq\Vert\phi_1\Vert_X^{Y,\om}+\Vert\phi_2\Vert_X^{Y,\om}\]
follows. This proves (\ref{eq:gh}).

Let $Y'$ as in the hypothesis of the {\bf second part of the proposition}, and $\phi\in\Ham(X,\om)$. The second statement is a consequence of the following claim. 
\begin{claim}\label{claim:eps H'} For every $\psi\in\Ham(Y,\om)$ satisfying $\psi|_X=\phi$ and every $\eps>0$, there exists $\psi'\in\Ham(Y',\om)$ such that
\begin{eqnarray}\label{eq:psi' Y}&\psi'|_X=\phi,&\\
\label{eq:Vert psi'}&\Vert\psi'\Vert^{Y',\om}<\Vert\psi\Vert^{Y,\om}+3\eps.&
\end{eqnarray}
\end{claim}
\begin{pf}[Proof of Claim \ref{claim:eps H'}] Assume that $\psi$ and $\eps$ are as above. We choose a function $H\in\HH(M,\om,Y)$ such that 
\begin{equation}\label{eq:phi H 1 Y}\phi_H^1|_Y=\psi,\quad\Vert H\Vert_Y<\Vert\psi\Vert^{Y,\om}+\eps.
\end{equation}
Since by hypothesis, $Y$ is compact and contained in $\Int Y'$, there exists a compact neighborhood $K_0$ of $Y$ that is contained in $\Int Y'$. We choose a compact neighborhood $K_1$ of $Y$ that is contained in $K_0$ and satisfies 
\begin{equation}\label{eq:max min}\max_{K_1}H^t\leq\max_YH^t+\eps,\quad\min_{K_1}H^t\geq\min_YH^t-\eps,\end{equation}
for every $t\in[0,1]$. Furthermore, we choose a compact neighborhood $K_2$ of $Y$ that is contained in $\Int K_1$. By Lemma \ref{le:Urysohn} below there exists a function $f\in C^\infty(M,[0,1])$ such that $f\const1$ on $K_2$ and $f\const0$ on $M\wo\Int K_1$. We choose a point $x_0\in Y$ and define $H':[0,1]\x M\to\R$ by 
\[H'(t,x):={H'}^t(x):=f(x)\big(H^t(x)-H^t(x_0)\big).\]
The support of this function is contained in $[0,1]\x K_1$ and hence compact. Hence its Hamiltonian flow exists on $M$. We define $\psi':=\phi_{H'}^1|_{Y'}$. It follows that $\psi'\in\Ham(Y',\om)$. For each $t\in[0,1]$, the functions ${H'}^t$ and $H^t$ differ on $K_2$ by the constant $H^t(x_0)$. Since $\phi_H^1|_Y=\psi$, $\psi|_X=\phi$ and $X\sub \Int K_2$, equality (\ref{eq:psi' Y}) follows.

Using that $f\leq1$, $f\const0$ on $M\wo K_1$, and inequalities (\ref{eq:max min}), we have, for every $t\in[0,1]$,
\begin{eqnarray*}\sup_{Y'}{H'}^t-\inf_{Y'}{H'}^t&\leq&\max_{K_1}(H^t-H^t(x_0))-\min_{K_1}(H^t-H^t(x_0))\\
&\leq&\max_YH^t-\min_YH^t+2\eps.
\end{eqnarray*}
It follows that 
\[\Vert H'\Vert_{Y'}\leq\Vert H\Vert_Y+2\eps.\]
Combining this with the inequality $\Vert\psi'\Vert^{Y',\om}\leq\Vert H'\Vert_{Y'}$ and the inequality in (\ref{eq:phi H 1 Y}), inequality (\ref{eq:Vert psi'}) follows. Hence $\psi'$ has the required properties. This proves Claim \ref{claim:eps H'} and hence the second statement, and completes the proof of Proposition \ref{prop:Vert X Y Y'}.
\end{pf}\end{proof}
\begin{proof}[Proof of Proposition \ref{prop:kkk}]\setcounter{claim}{0} {\bf Inequality (\ref{eq:kkk n d})} follows by taking $\ell:=2n-d$, $k_i:=1$, for $i=1,\ldots,\ell$, $n_i:=1$, for $i=1,\ldots,\ell-1$, and $n_\ell:=d-n+1$. 

Let $n\in\N$ be such that $n\neq k^2$, for every $k\in\N$. {\bf Inequality ``$\geq$'' in (\ref{eq:k K})} is a consequence of the next claim. Let $\ell\geq2$ and $k_1,\ldots,k_\ell$ be as in the definition of $\kkk(n,n)$.
\begin{claim}\label{claim:sum k i 2 n} We have
\begin{equation}\label{eq:sum k i 2 n}\sum_{i=1}^\ell k_i^2=n.
\end{equation}
\end{claim}
\begin{proof}[Proof of Claim \ref{claim:sum k i 2 n}] We choose integers $n_1,\ldots,n_\ell$ such that the inequalities (\ref{eq:n i k i},\ref{eq:k i n i d geq},\ref{eq:2 min}) are satisfied. Subtracting the first from the second inequality in (\ref{eq:k i n i d geq}), we obtain $\sum_ik_i(n_i-k_i)\leq0$. Using the inequalities (\ref{eq:n i k i}), it follows that $n_i=k_i$, for every $i=1,\ldots,\ell$. Combining this with (\ref{eq:k i n i d geq}), the equality (\ref{eq:sum k i 2 n}) follows. This proves Claim \ref{claim:sum k i 2 n}.
\end{proof}

We show {\bf inequality ``$\leq$'' in (\ref{eq:k K})}: Let $\ell\in\N$ and $k_1,\ldots,k_\ell\in\N$ be as in the definition of $\K(n)$. This means that $\sum_{i=1}^\ell k_i^2=n$. Our hypothesis that $n\neq k^2$, for every $k\in\N$, implies that the condition (\ref{eq:ell 2}) is satisfied.

We define $n_i:=k_i$, for $i=1,\ldots,\ell$. The conditions (\ref{eq:n i k i},\ref{eq:k i n i d geq}) are satisfied with $d=n$. Furthermore, using that $\ell\geq2$, it follows that (\ref{eq:2 min}) holds. Inequality ``$\leq$'' in (\ref{eq:k K}) follows. This proves (\ref{eq:k K}).

{\bf Inequality (\ref{eq:K})} was proved in \cite{SZSmall} (Proposition 8, inequality (36)). This completes the proof of Proposition \ref{prop:kkk}.
\end{proof}
\begin{proof}[Proof of Proposition \ref{prop:Diam diam}]\setcounter{claim}{0} Assume that $c\in\big[0,\diam(X,M,\om)\big)$. By definition there exists $\phi\in\Ham(X,\om)$ such that $\Vert\phi\Vert_X^{M,\om}\geq c$. We choose a function $H\in\HH(M,\om,X)$ such that $\phi_H^1|_X=\phi$. We also choose a  function $\rho\in C^\infty(\R^{2n},[0,1])$ with compact support contained in $U$, such that $\rho\const1$ in some neighborhood $V\sub M$ of $X$. We define $\wt H:[0,1]\x M\to\R$ by $\wt H(t,x):=\rho(x)H(t,x)$. 

Note that the support of $\wt H$ is compact and contained in $U$. Furthermore, we have $\phi_{\wt H}^1|_X=\phi$. It follows that 
\[\Diam(M,\om,U)\geq\Vert\phi_{\wt H}^1\Vert_M\geq\Vert\phi\Vert_X^{M,\om}\geq c.\]
Since $c<\diam(X,M,\om)$ is arbitrary, the inequality (\ref{eq:Diam diam}) follows. This proves Proposition \ref{prop:Diam diam}.
\end{proof}
\section{Coisotropic intersections and relative Hofer diameters}\label{sec:coiso Hofer}
This section is the core of the article. We develop a framework for proving a lower bound on the relative Hofer diameter of a set. We use this to prove the main result, Theorem \ref{thm:exists X}, in Section \ref{subsec:proof:thm:exists X}. The method described here is of interest in its own, since it can be used to prove similar results in different settings. 
\subsection{Coisotropic intersections}\label{subsec:coiso}
The proof of Theorem \ref{thm:exists X} is based on the following result about coisotropic intersections, which we proved in \cite{SZSmall}. To state it, let $(M,\om)$ be a symplectic manifold. We call it \emph{(weakly geometrically) bounded} \label{bounded} iff there exist an almost complex structure $J$ on $M$ and a complete Riemannian metric $g$ such that the following conditions hold:
\begin{itemize}
\item The sectional curvature of $g$ is bounded and $\inf_{x\in M}\iota^g_x>0$, where $\iota^g_x$ denotes the injectivity radius of $g$ at the point $x\in M$. 
\item There exists a constant $C\in(0,\infty)$ such that 
\[|\om(v,w)|\leq C|v|\,|w|,\quad\om(v,Jv)\geq C^{-1}|v|^2,\]
for every $v,w\in T_xM$ and $x\in M$. Here $|v|:=\sqrt{g(v,v)}$.
\end{itemize}
This is a mild condition on $(M,\om)$. (For examples see \cite{SZSmall}.)

Recall that a submanifold $N$ of $M$ is called \emph{coisotropic} iff for $x\in N$ the subspace
\[T_xN^\om=\big\{v\in T_xM\,\big|\,\om(v,w)=0,\,\forall w\in T_xN\big\}\]
of $T_xM$ is contained in $T_xN$. As an example, $N$ is coisotropic if it is a hypersurface.

Let $N\sub M$ be a coisotropic submanifold. We denote
\[N_\om:=\big\{\textrm{isotropic leaf of }N\big\},\]
and define the \emph{action spectrum} and the \emph{minimal action} of $(M,\om,N)$ as
\begin{equation}\label{eq:S}S(M,\om,N):=\left\{\int_\D u^*\om\,\bigg|\,u\in C^\infty(\D,M):\,\exists F\in N_\om:\, u(S^1)\sub F\right\},\end{equation}
\begin{equation}\label{eq:A M om N} A(M,\om,N):=\inf\big(S(M,\om,N)\cap (0,\infty)\big)\in[0,\infty].
\end{equation}
(Here our convention is that $\inf\emptyset:=\infty$.) We define the \emph{split minimal symplectic action of $N$, $A_{\Cross}(M,\om,N)$} as follows. We define a \emph{bounded splitting of $(M,\om,N)$} to be a tuple %%% check new line %%% keep comment
$(M_i,\om_i,N_i)_{i=1,\ldots,k}$, where $k\in\N$ and for every $i=1,\ldots,k$, $(M_i,\om_i)$ is a bounded symplectic manifold and $N_i\sub M_i$ a coisotropic submanifold, such that there exists a symplectomorphism $\phi$ from $\big(\Cross_{i=1}^kM_i,\oplus_{i=1}^k\om_i\big)$ to $(M,\om)$, satisfying $\phi\big(\Cross_{i=1}^kN_i\big)=N$.

We define 
\begin{equation}\label{eq:A Cross}A_\Cross(N):=A_{\Cross}(M,\om,N):=
\end{equation}
\[\sup\big\{\min_{i=1,\ldots k}A(M_i,\om_i,N_i)\,\big|\,(M_i,\om_i,N_i)_i\textrm{ bounded splitting of }(M,\om,N)\big\}.\]
Here our convention is that $\sup\emptyset=0$. 

\noindent{\bf Remark.}~If $(M,\om)$ is not bounded then $(M,\om,N)$ does not admit any bounded splitting, and therefore $A_{\Cross}(M,\om,N)=0$. This follows from the facts that a finite product of bounded symplectic manifolds is bounded, and boundedness is invariant under symplectomorphisms. $\Box$

We call a coisotropic submanifold $N\sub M$ \emph{regular} iff its isotropy relation is a closed subset and a submanifold of $N\x N$. Equivalently, the symplectic quotient of $N$ is well-defined. (For more details and examples see \cite{ZiLeafwise}.) We abbreviate $\HH(M,\om):=\HH(M,\om,M)$ and define the \emph{displacement energy} of a subset $X\sub M$ to be 
\begin{equation}\label{eq:e X M om}e(X,M):=e(X,M,\om):=\inf\big\{\Vert H\Vert\,\big|\,H\in\HH(M,\om)\,\big|\,\phi_H^1(X)\cap X=\emptyset\big\}.
\end{equation}
We call a manifold \emph{closed} iff it is compact and its boundary is empty. We are now able to formulate the key result. 
\begin{thm}[Coisotropic intersections, see \cite{SZSmall} (Theorem 1)]\label{thm:N phi N}Let $(M,\om)$ be a symplectic manifold and $\emptyset\neq N\sub M$ a closed connected regular coisotropic submanifold. Then we have
\begin{equation}\label{eq:e A}e(N,M)\geq A_{\Cross}(N).
\end{equation}
\end{thm}
The proof of this result (see \cite{SZSmall}) is based on a certain Lagrangian embedding of $N$ and on the Main Theorem in the article \cite{Ch} by Y.~Chekanov.

The idea of proof for part (\ref{thm:exists X:S}) of Theorem \ref{thm:exists X} is to find a Hamiltonian flow $[0,1]\x\R^{2n}\ni(t,x)\mapsto\phi^t(x)\in\R^{2n}$ that preserves $S^{2n-1}$, such that the following holds. Let $({\phi'}^t)$ be a Hamiltonian flow generated by some function in $\HH(\R^{2n},\om_0,S^{2n-1})$, such that ${\phi'}^1|_{S^{2n-1}}=\phi^1|_{S^{2n-1}}$. Then there exists a regular closed coisotropic submanifold $\emptyset\neq N\sub\R^{2n}$ such that
\[A_\Cross(N)\geq\frac\pi2,\quad{\phi'}^1(N)\cap N=\emptyset.\]
It then follows from Theorem \ref{thm:N phi N} that 
\[\Vert\phi^1|_{S^{2n-1}}\Vert^{\R^{2n},\om_0}_{S^{2n-1}}\geq\frac\pi2.\]
The claimed inequality (\ref{eq:diam S 2n - 1}) is a consequence of this. 

In the following subsection we will put this idea into a more general framework, which we will use for the proofs of both parts of Theorem \ref{thm:exists X}.
\subsection{Rigidifying pairs}\label{subsec:rigid}
Let $(M,\om)$ be a symplectic manifold. In this subsection, given a compact subset of $M$ and a Hamiltonian $S^1$-action on $M$, we construct a pair $(X,\phi)$, where $X\sub M$, and $\phi\in\Ham(X,\om)$, and we prove a lower bound on the Hofer norm of $\phi$ on $X$ relative to $M$. This is a key ingredient of the proof of Theorem \ref{thm:exists X}.

Let $X_0\sub M$ be a subset, and $\al\in\Om^1(M)$. We call the pair $(X_0,\al)$ \emph{rigidifying} iff for every symplectomorphism $\phi:M\to M$ the following holds. If $\phi|_{X_0}=\id_{X_0}$ then, for every $x\in X_0$ and $v\in T_xM$, we have
\begin{equation}\label{eq:al phi x}\al\,d\phi v=\al v.
\end{equation}
As an example, $(X_0,\al)$ is rigidifying if $X_0$ is open.

Let $S^1\x M\ni(z,x)\mapsto\phi^z(x)\in M$ be a Hamiltonian action. We fix a compact subset $X_0\sub M$, and define
\begin{equation}\label{eq:X}X:=\bigcup_{z\in S^1}\phi^z(X_0).\end{equation}
This is a compact subset of $M$. Let $z_0\in S^1\sub\C$. We denote by $\BAR{z_0}$ its complex conjugate. Note that $\phi^{\BAR{z_0}}|_X\in\Ham(X,\om)$. The next result gives a lower bound on $\Vert\phi^{\BAR{z_0}}|_X\Vert_X^{M,\om}$ (defined as in (\ref{eq:Vert X Y om})), if there exists a suitable rigidifying one-form $\al$ for $X_0$. Recall the definition (\ref{eq:e X M om}) of the displacement energy $e(X,M)=e(X,M,\om)$ of a subset $X\sub M$. 
\begin{lem}[Main Lemma]\label{le:d X limsup} Let $[0,1]\x M\ni(t,x)\mapsto\psi^t(x)\in M$ be a smooth map satisfying
\begin{equation}\label{eq:psi 0 id}\psi^0=\id.
\end{equation}
Assume that there exists a function $f\in C^\infty(M,\R)$ such that the pair $(X_0,df)$ is rigidifying, and
\begin{eqnarray}\label{eq:X phi Bar z 0}&f(X_0)\sub[0,\infty),\quad f\circ\phi^{z_0}(X_0)\sub(-\infty,0],&\\ 
\label{eq:dd t f psi t}&\left.\frac d{dt}\right|_{t=0}(f\circ\psi^t(x))>0,\,\,\left.\frac d{dt}\right|_{t=0}(f\circ\phi^{z_0}\circ\psi^t(x))\leq0,\,\,\forall x\in X_0.&
\end{eqnarray}
Then we have (with $X$ as in (\ref{eq:X}))
\begin{equation}\label{eq:d X limsup}\Vert\phi^{\BAR{z_0}}|_X\Vert_X^{M,\om}\geq\limsup_{t\searrow0}e(\psi^t(X_0),M,\om). 
\end{equation}
\end{lem}
\begin{proof}[Proof of Lemma \ref{le:d X limsup}]\setcounter{claim}{0} We choose an open subset $U\sub M$ containing $X$ such that $\BAR{U}$ is compact. Let $\Phi\in\Ham(M,\om)$ be such that 
\begin{equation}\label{eq:psi X phi}\Phi|_X=\phi^{\BAR{z_0}}|_X.\end{equation} 
\begin{claim}\label{claim:d X M om} There exists $t_0\in(0,1]$ such that for $t\in(0,t_0]$, we have 
\[\Phi\circ\psi^t(X_0)\cap\psi^t(X_0)=\emptyset.\]
\end{claim}
\begin{proof}[Proof of Claim \ref{claim:d X M om}] We define $\wt f:=f\circ\phi^{z_0}$. We check the hypotheses of Lemma \ref{le:X 0 phi psi} (below) with $\phi:=\Phi$ and $f$ replaced by $\wt f$: The inclusions (\ref{eq:f X 0}) follow from (\ref{eq:X phi Bar z 0},\ref{eq:psi X phi}). 

We prove the inequalities (\ref{eq:f psi t x}): The first inequality in (\ref{eq:f psi t x}) follows from the second inequality in (\ref{eq:dd t f psi t}). Let $x_0\in X_0$. We define $x(t):=\psi^t(x)$. To prove the second inequality (with $x$ replaced by $x_0$), observe that by (\ref{eq:psi 0 id}), we have $x(0)=x_0$. Furthermore, (\ref{eq:psi X phi}) implies that $\phi^{z_0}\circ\Phi|_{X_0}=\id_{X_0}$. Therefore the hypothesis that $(X_0,df)$ is rigidifying implies that 
\begin{equation}\label{eq:df d phi z 0}df\,d(\phi^{z_0}\circ\Phi)\dot x(0)=df\,\dot x(0).
\end{equation}
The left hand side of this equality equals $\left.\frac d{dt}\right|_{t=0}\big(\wt f\circ\Phi\circ x\big)$. Furthermore, by the first inequality in (\ref{eq:dd t f psi t}), the right-hand side of (\ref{eq:df d phi z 0}) is positive. Hence the second inequality in (\ref{eq:f psi t x}) Therefore, all the hypotheses of Lemma \ref{le:X 0 phi psi} are satisfied. Applying this lemma, the statement of Claim \ref{claim:d X M om} follows.
\end{proof}
Claim \ref{claim:d X M om} implies that 
\begin{equation}\label{eq:d limsup} 
\Vert\Phi\Vert^{M,\om}\geq\limsup_{t\searrow0}e(\psi^t(X_0),M,\om).
\end{equation}
Since this holds for every $\Phi\in\Ham(M,\om)$ satisfying (\ref{eq:psi X phi}), inequality (\ref{eq:d X limsup}) follows. This proves Lemma \ref{le:d X limsup}. 
\end{proof}
The next result provides a large class of examples of rigidifying pairs $(X,\al)$. Let $X,Y,Y'$ be smooth manifolds and $f\in C^\infty(X,Y)$ and $f'\in C^\infty(X,Y')$ maps. We say that $f$ \emph{factors} through $f'$ iff there exists a map $g\in C^\infty(Y',Y)$ satisfying $f=g\circ f'$. Let $(M,\om)$ be a symplectic manifold, $X_0\sub M$ a subset, and $f\in C^\infty(M,\R)$. 
\begin{lem}\label{le:rigid} The pair $(X_0,df)$ is rigidifying, provided that there exist symplectic manifolds $(\wt M,\wt\om)$ and $(M',\om')$, a connected Lie group $G$, a Hamiltonian action of $G$ on $\wt M$, an moment map $\mu:\wt M\to\g^*$ for the action, and a symplectomorphism $\psi:\wt M\to M\x M'$ such that the following holds. (Note that by definition, $\mu$ is equivariant.) We denote by $\pr:M\x M'\to M$ the canonical projection. Then the composition $f\circ\pr\circ\psi$ factors through $\mu$, and we have
\begin{equation}\label{eq:X pr mu}X_0=\pr\circ\psi(\mu^{-1}(0)). 
\end{equation}
\end{lem}
The proof of Lemma \ref{le:rigid} is based on the following result. Let $(M,\om)$ be a symplectic manifold, $\phi$ a symplectomorphism on $M$, and $G$ a connected Lie group. We denote by $\g$ the Lie algebra of $G$ and fix a Hamiltonian action of $G$ on $M$ and an (equivariant) moment map $\mu:M\to\g^*$. 
\begin{lem}\label{le:d mu phi} Assume that $\phi(\mu^{-1}(0))=\mu^{-1}(0)$ and the restriction of $\phi$ to $\mu^{-1}(0)$ is $G$-equivariant. Then we have
\begin{equation}\label{eq:d mu phi}d(\mu\circ\phi)(x)=d\mu(x),\quad\forall x\in \mu^{-1}(0).
\end{equation}
\end{lem}
In the proof of this lemma, for $\xi\in\g$ we denote by $X_\xi$ the vector field on $M$ generated by $\xi$. 

\begin{proof}[Proof of Lemma \ref{le:d mu phi}]\setcounter{claim}{0} Let $x\in\mu^{-1}(0)$, $v\in T_xM$ and $\xi\in\g$. Then we have
\begin{equation}\label{eq:d mu d phi}\lan d\mu(\phi(x))d\phi(x)v,\xi\ran=\om\big(X_\xi\circ\phi(x),d\phi(x)v\big).\end{equation}
Since, by assumption, $\phi|_{\mu^{-1}(0)}$ is $G$-equivariant, we have $\phi\big(\exp(t\xi)x\big)=\exp(t\xi)\phi(x)$, for every $t\in\R$. Taking the derivative at $t=0$, it follows that 
\[d\phi(x)X_\xi(x)=X_\xi\circ\phi(x).\]
Combining this with equality (\ref{eq:d mu d phi}) and using that $\phi$ is a symplectomorphism, it follows that 
\[\lan d\mu(\phi(x))d\phi(x)v,\xi\ran=\lan d\mu(x)v,\xi\ran.\]
Equality (\ref{eq:d mu phi}) follows. This proves Lemma \ref{le:d mu phi}.
\end{proof}
\begin{proof}[Proof of Lemma \ref{le:rigid}]\setcounter{claim}{0}Let $\wt M$ etc. be as in the hypothesis. By a straight-forward argument, we may assume without loss of generality that $\wt M=M\x M'$ and $\psi$ is the identity map on $\wt M$. Let $\phi:M\to M$ be a symplectomorphism satisfying 
\begin{equation}\label{eq:phi X 0 id}\phi|_{X_0}=\id_{X_0}.
\end{equation}
In order to show that equality (\ref{eq:al phi x}) holds, we define $\wt\phi:=\phi\x\id_{M'}:\wt M\to\wt M$. By hypothesis there exists a map $g\in C^\infty(\g^*,\R)$ such that 
\begin{equation}\label{eq:f pr g mu}f\circ\pr=g\circ\mu.\end{equation} 
Using that $\phi\circ\pr=\pr\circ\wt\phi$, it follows that
\begin{equation}\label{eq:d f phi}d(f\circ\phi)d\pr=dg\,d(\mu\circ\wt\phi). 
\end{equation}
The map $\wt\phi$ is an $\om\oplus\om'$-symplectomorphism. Furthermore, equalities (\ref{eq:phi X 0 id},\ref{eq:X pr mu}) imply that $\wt\phi|_{\mu^{-1}(0)}=\id_{\mu^{-1}(0)}$. Therefore, we may apply Lemma \ref{le:d mu phi} with $\phi$ replaced by $\wt\phi$, and conclude that
\[d(\mu\circ\wt\phi)(\wt x)=d\mu(\wt x),\quad\forall\wt x\in\mu^{-1}(0).\]
Combining this with equalities (\ref{eq:d f phi},\ref{eq:f pr g mu}), we obtain
\[d(f\circ\phi)d\pr(\wt x)=df\,d\pr(\wt x),\quad\forall\wt x\in\mu^{-1}(0).\]
Using (\ref{eq:X pr mu}) and that $\pr$ is submersive, it follows that $df\,d\phi(x)=df(x)$, for every $x\in X_0$. It follows that $(X_0,df)$ is rigidifying. This proves Lemma \ref{le:rigid}.
\end{proof}
\subsection{Proof of Theorem \ref{thm:exists X} (Relative Hofer diameter of a small subset of a symplectic manifold)}\label{subsec:proof:thm:exists X}
Both parts of this result are proved along similar lines. The idea for the first part is to define $X_0$ to be the product of a circle and a sphere in $\R^{2n-2}$, each of radius $1/\sqrt2$, $\phi$ a certain linear unitary action of $S^1$ on $\R^{2n}$, $X:=\bigcup_{z\in S^1}\phi^z(X_0)$, and $\psi^t:\R^{2n}\to\R^{2n}$ a map that expands the circle-factor by $(1+t)$. It follows that $X\sub S^{2n-1}$. We may then apply Lemmas \ref{le:rigid} and \ref{le:d X limsup}, obtaining inequality (\ref{eq:d X limsup}). 

Since $\psi^t(X_0)$ is a regular coisotropic submanifold of $\R^{2n}$, we may then use the key result, Theorem \ref{thm:N phi N}, to estimate the right-hand side of inequality (\ref{eq:d X limsup}) from below by $\frac\pi2$. The claimed inequality (\ref{eq:diam S 2n - 1}) is a consequence of this and the following remark. 
\begin{rem}\label{rem:X Y} Let $(M,\om)$ be a symplectic manifold, $X\sub Y\sub M$ closed subsets and $H\in C^\infty\big([0,1]\x M,\R\big)$ a function such that $Y\sub\DD_{X_H}^1$ (the domain of the flow $\phi_H^1$) and $\phi_H^t(X)=X$, $\phi_H^t(Y)=Y$, for every $t\in[0,1]$. Then we have
\[\Vert\phi_H^1|_X\Vert_X^{M,\om}\leq\Vert\phi_H^1|_Y\Vert_Y^{M,\om}.\]
This follows from a straight-forward argument. $\Box$
\end{rem}
\begin{proof}[Proof of Theorem \ref{thm:exists X}(\ref{thm:exists X:S})]\setcounter{claim}{0} For $k\in\N$ and $a>0$ we denote by $S^{2k-1}(a)\sub\R^{2k}$ the sphere of radius $\sqrt{a/\pi}$ around 0. We define
\[X_0:=S^1(\frac\pi2)\x S^{2n-3}(\frac\pi2),\]
(Here we use the hypothesis that $n\geq2$.) Furthermore, we define the map 
\begin{equation}\label{eq:S 1 R 2n}S^1\x\R^{2n}\ni(z,x)\mapsto\phi^z(x)\in\R^{2n}\end{equation}
as follows. Let $z\in S^1$. We denote by $R^z:\R^2\to\R^2$ the rotation by $z$. (Identifying $\R^2=\C$, it is given by the formula $\phi^z(q_1+iq_2):=z(q_1+iq_2)$.) We define $\phi^z:\R^{2n}=\C^n=\C^2\x\C^{n-2}\to\C^n$ to be the unique complex linear extension of the map 
\[R^z\x\id_{\R^{n-2}}:\R^n=\R^2\x\R^{n-2}\to\R^n.\] 
(Note that the identification of $\R^{2n}$ with $\C^n$ here is not compatible with the identification of $\R^2$ with $\C$ in the above formula for $\phi^z$.) The map (\ref{eq:S 1 R 2n}) is a Hamiltonian $S^1$-action on $\C^n$. (It is generated by the function $H:\C^n\to\R$ defined by $H(q+ip):=q_1p_2-q_2p_1$.) We define $X:=\bigcup_{z\in S^1}\phi^z(X_0)$. Since $X_0\sub S^{2n-1}$ and $\phi^z$ is orthonormal, for every $z\in S^1$, it follows that $X\sub S^{2n-1}$, and $\phi^z$ preserves $X$ and $S^{2n-1}$, for every $z\in S^1$. Therefore, by Remark \ref{rem:X Y}, we have 
\begin{equation}\label{eq:d S 2n - 1}\Vert\phi^{-i}|_{S^{2n-1}}\Vert_{S^{2n-1}}^{\R^{2n},\om_0}\geq\Vert\phi^{-i}|_X\Vert_X^{\R^{2n},\om_0}.
\end{equation}
We define 
\begin{eqnarray*}&(M,\om):=(\C^n,\om_0),\quad z_0:=i,&\\
&\psi:[0,1]\x\C^n\to\C^n,\quad\psi^t(y,y'):=\psi(t,y,y'):=((1+t)y,y'),&\end{eqnarray*}
for $t\in[0,1]$ and $(y,y')\in\C^n=\C\x\C^{n-1}$. 
\begin{claim}\label{claim:hyp} The hypotheses of Lemma \ref{le:d X limsup} are satisfied. 
\end{claim}
\begin{proof}[Proof of Claim \ref{claim:hyp}] The {\bf condition (\ref{eq:psi 0 id})} is clearly satisfied. We define the map 
\[f:\C^n=\C\x\C^{n-1}\to\R,\quad f(y,y'):=|y|^2-\frac12.\]
By the next claim, we may apply Lemma \ref{le:rigid}, to conclude that {\bf $(X_0,df)$ is rigidifying}. 
\begin{claim}\label{claim:hyp le:rigid} The pair $(X_0,f)$ satisfies the hypotheses of Lemma \ref{le:rigid}. 
\end{claim}
\begin{proof}[Proof of Claim \ref{claim:hyp le:rigid}] We define $(\wt M,\wt\om):=(\C^n,\om_0)$, $M'$ to be a point, $G:=S^1\x S^1$, the action of $G$ on $\C^n=\C\x\C^{n-1}$ by $(z,z')\cdot(y,y'):=(zy,z'y')$, the moment map 
\[\mu(y,y'):=\frac i2\left(\frac12-|y|^2,\frac12-|y'|^2\right)\in i\R\x i\R=\g\iso\g^*,\]
and $\psi$ to be the identity on $\wt M=M\x M'$. Then the hypotheses Lemma \ref{le:rigid} are satisfied. This proves Claim \ref{claim:hyp le:rigid}.
\end{proof}
The next hypothesis of Lemma \ref{le:d X limsup}, the {\bf inclusions (\ref{eq:X phi Bar z 0})}, follow from the facts 
\begin{equation}\label{eq:f X 0 0 f phi}f(X_0)=\{0\},\quad f\circ\phi^i(X_0)\sub[-\frac12,0].
\end{equation} 
(Here in the second condition we used that $\phi^i(y,y'',y''')=(-y'',y,y''')$, for every $(y,y'',y''')\in\C\x\C\x\C^{n-2}=\C^n$.) We prove that the {\bf inequalities (\ref{eq:dd t f psi t})} are satisfied: Direct calculations show that
\[\left.\frac d{dt}\right|_{t=0}\big(f\circ\psi^t(y,y')\big)=2|y|^2,\quad\left.\frac d{dt}\right|_{t=0}\big(f\circ\phi^i\circ\psi^t(y,y')\big)=0,\]
for every $(y,y')\in\C^n=\C\x\C^{n-1}$. Since every $x=(y,y')\in X_0$ satisfies $2|y|^2=1$, (\ref{eq:dd t f psi t}) follows. Hence all the hypotheses of Lemma \ref{le:d X limsup} are satisfied. This proves Claim \ref{claim:hyp}.
\end{proof}
By Claim \ref{claim:hyp} we may apply Lemma \ref{le:d X limsup}, to conclude that inequality (\ref{eq:d X limsup}) holds. 

Let now $t\in[0,1]$. Then we have $\psi^t(X_0)=S^1\big((1+t)^2\frac\pi2\big)\x S^{2n-3}(\frac\pi2)$. This is a closed regular coisotropic submanifold of $\R^{2n}$. Therefore, applying Theorem \ref{thm:N phi N}, inequality (\ref{eq:e A}) holds with $N:=\psi^t(X_0)$. Remark \ref{rmk:S N N'} and Proposition \ref{minimalareaStiefel} below imply that
\[A_\Cross\big(\R^{2n},\om_0,\psi^t(X_0)\big)\geq \min\big\{(1+t)^2\frac\pi2,\frac\pi2\big\}=\frac\pi2.\]
Combining this with (\ref{eq:d S 2n - 1},\ref{eq:d X limsup},\ref{eq:e A}), it follows that 
\[\Vert\phi^{-i}|_{S^{2n-1}}\Vert_{S^{2n-1}}^{\R^{2n},\om_0}\geq\frac\pi2.\]
Recalling the definition (\ref{eq:diam X Y om}) of $\diam(X,Y,\om)$, inequality (\ref{eq:diam S 2n - 1}) follows. This proves statement (\ref{thm:exists X:S}) of Theorem \ref{thm:exists X}. 
\end{proof}
\begin{rem}\label{rmk:} In \cite{ZiLeafwise} the second author proved a result (Theorem 1) similar to Theorem \ref{thm:N phi N}. That result states a positive lower bound on the number of leafwise fixed points of the given Hamiltonian diffeomorphism. Hence its conclusion is stronger than that of Theorem \ref{thm:N phi N}. 

The hypotheses of both results are the same, except that in \cite[Theorem 1]{ZiLeafwise} it is assumed that $\Vert\phi\Vert_\om<A(N)$, rather than $\Vert\phi\Vert_\om<A_\Cross(N)$. (The former condition is simpler and stronger than the latter.) 

Since for $\psi^t$ and $X_0$ as in the above proof of Theorem \ref{thm:exists X}(\ref{thm:exists X:S}), we have 
\[A(\psi^t(X_0))\to0,\quad\textrm{as }t\searrow0,\]
\cite[Theorem 1]{ZiLeafwise} is not suitable for this proof. We really need the refinement given in the present article. The same holds for the proof of Theorem \ref{thm:exists X}(\ref{thm:exists X:k'}) (see below). 
\end{rem}
{\bf Outline of the proof of the second part of Theorem \ref{thm:exists X}:} This is a refinement of the technique used in the proof of part (\ref{thm:exists X:S}). The idea is as follows: We choose $\ell\geq2$ and $k_1,\ldots,k_\ell\in\N$, such that 
\[\sum_{i=1}^\ell k_i=\kkk(n,d),\]
and there exist integers $n_1,\ldots,n_\ell$ as in the definition of $\kkk(n,d)$. Without loss of generality we may assume that $n_1=\min_in_i$. Assume first that 
\begin{equation}\label{eq:sum n}\sum_{i=1}^\ell k_in_i=n.\end{equation}
For every pair $k,m\in\N$ satisfying $k\leq m$, and $a>0$, we define the \emph{Stiefel manifold of area $a$} to be
\[V(k,n,a):=\big\{\Theta\in\C^{k\x n}\,\big|\,\Theta\Theta^*=\frac a\pi\one_k\big\}.\]
We define $a:=\frac\pi{\kkk(n,d)}$ and 
\[X_0:=\Cross_{i=1}^\ell V(k_i,n_i,a)\sub\Cross_{i=1}^\ell\C^{k_i\x n_i}=\C^n.\]
Here rescaling the standard Stiefel manifolds ensures that $X_0\sub\BAR{B}^{2n}$. The second part of condition (\ref{eq:k i n i d geq}) guarantees that the dimension of $X_0$ is bounded above by $d$. 

We choose a linear unitary action $S^1\x\C^n\ni(z,x)\mapsto\phi^z(x)\in\C^n$ which for a given tuple of matrices $(\Theta_1,\ldots,\Theta_\ell)\in X_0$ intertwines the first row of $\Theta_1$ with part of the first row of $\Theta_2$. (This makes sense because of our assumption that $n_1=\min_in_i$.) The set $X:=\bigcup_{z\in S^1}\phi^z(X_0)$ has the properties required in statement (\ref{thm:exists X:k'}): That $X\sub\BAR{B}^{2n}$ follows from the fact $X_0\sub\BAR{B}^{2n}$ and the orthogonality of the action $\phi$. Furthermore, since $\dim X_0\leq d$, the Hausdorff dimension of $X$ is bounded above by $d+1$. 

The main task is to show that inequality (\ref{eq:diam X}) holds. We will prove this by showing that the restriction of the map $\phi^{-i}:\C^n\to\C^n$ to $X$ has relative Hofer semi-norm bounded below by the right-hand side of (\ref{eq:diam X}). The proof of this bound is based on the Lemmas \ref{le:d X limsup} and \ref{le:rigid}. The remainder of the argument is now analogous to the argument for part (\ref{thm:exists X:S}).

If the integers $\ell$, $k_1,\ldots,k_\ell$, $n_1,\ldots,n_\ell$ cannot be chosen such that the equality (\ref{eq:sum n}) holds, then the idea is to ``project away'' the extra $(\sum_ik_in_i)-n$ complex dimensions. This means that we construct a suitable surjective linear map
\[\Psi:\Cross_{i=1}^\ell\C^{k_i\x n_i}\to\C^n,\]
and define 
\[X_0:=\Psi\big(\Cross_{i=1}^\ell V(k_i,n_i,a)\big)\sub\C^n.\]
We may then carry out a modified version of the above argument. 
\begin{proof}[Proof of Theorem \ref{thm:exists X}(\ref{thm:exists X:k'})]\setcounter{claim}{0} We choose $\ell\in\{2,3,\ldots\}$ and $k_1,\ldots,k_\ell\in\N$ as in the definition of $\kkk(n,d)$ such that 
\begin{equation}\label{eq:sum k}\sum_{i=1}^\ell k_i=\kkk(n,d).
\end{equation} 
We also choose integers $n_1,\ldots,n_\ell$ such that the conditions (\ref{eq:n i k i},\ref{eq:k i n i d geq},\ref{eq:2 min}) are satisfied. Reordering the pairs $(k_i,n_i)$, we may assume that $n_1=\min_in_i$. We choose an injective map 
\[(\lf,\kf,\nf):\{n_1+n_2+1,\ldots,n\}\to\{1,\ldots,\ell\}\x\N\x\N\]
satisfying  
\begin{eqnarray}\label{eq:K i}&\kf(i)\leq k_{\lf(i)},\quad\nf(i)\leq n_{\lf(i)},&\\
\label{eq:lf kf i}&(\lf,\kf)(i)\neq(1,1)\textrm{ or }(2,1),&\end{eqnarray}
for every $i\in\{n_1+n_2+1,\ldots,n\}$. (Our convention is that $\{m,\ldots,n\}:=\emptyset$ if $m>n$.) To see that we may choose this map to be injective, note that the number of allowed choices of $(\lf,\kf,\nf)$ is 
\[(k_1-1)n_1+(k_2-1)n_2+\sum_{L=3}^\ell k_Ln_L=\sum_{L=1}^\ell k_Ln_L-n_1-n_2.\]
(The first and second term are obtained by considering $L=1,2$, and the other terms by considering $L\geq3$.) By the first condition in (\ref{eq:k i n i d geq}) the right-hand side of this equality is bounded below by 
\[n-n_1-n_2=\big|\big\{n_1+n_2+1,\ldots,n\big\}\big|.\]
It follows that we may choose the map $(\lf,\kf,\nf)$ to be injective. We extend $(\lf,\kf,\nf)$ to $\{1,\ldots,n\}$ by defining 
\begin{eqnarray}\label{eq:L K N n1}&(\lf,\kf,\nf)(i):=(1,1,i),\quad\forall i\in\{1,\ldots,n_1\},&\\
\label{eq:L K N n2}&(\lf,\kf,\nf)(i):=(2,1,i-n_1),\,\forall i\in\big\{n_1+1,\ldots,\min\{n_1+n_2,n\}\big\}.
\end{eqnarray}
(Since $2n_1=2\min_in_i\leq n$, (\ref{eq:L K N n1}) makes sense.) We define the map $\Psi:\Cross_{i=1}^\ell\C^{k_i\x n_i}\to\C^n$ by
\begin{equation}\label{eq:Psi Cross}\Psi^i(\Theta_1,\ldots,\Theta_\ell):=(\Theta_{\lf(i)})^{\kf(i)}_{\nf(i)},\,\forall i\in\{1,\ldots,n\},\end{equation}
where for a matrix $\Theta$ the number $\Theta^i_j\in\C$ denotes its $(i,j)$-th entry. It follows from the inequalities (\ref{eq:K i}) (holding for $i\in\{n_1+n_2+1,\ldots,n\}$) and (\ref{eq:L K N n1},\ref{eq:L K N n2}) that this definition makes sense. We define 
\begin{eqnarray}\label{eq:a}&a:=\frac\pi{\kkk(n,d)},&\\
\label{eq:X 0 Psi}&X_0:=\Psi\big(\Cross_{i=1}^\ell V(k_i,n_i,a)\big)\sub\C^n.&\end{eqnarray}
Furthermore, we define the map 
\begin{equation}\label{eq:phi z }S^1\x\C^n\ni(z,x)\mapsto\phi^z(x)\in\C^n
\end{equation} 
as follows. Let $z\in S^1$. We denote by $R^z:\R^2\to\R^2$ the rotation by $z$, and define 
\[T^z:\R^{n_1}\x\R^{n_1}\to\R^{n_1}\x\R^{n_1},\quad T^z(q,q'):=(Q,Q'),\] 
where $(Q_i,Q'_i):=R^z(q_i,q'_i)$, for every $i\in\{1,\ldots,n_1\}$. We define $\phi^z:\C^n=\C^{2n_1}\x\C^{n-2n_1}\to\C^n$ to be the unique complex linear extension of the map 
\[T^z\x\id_{\R^{n-2n_1}}:\R^n=\R^{2n_1}\x\R^{n-2n_1}\to\R^n.\]
(The conditions (\ref{eq:2 min}) and $n_1=\min_in_i$ guarantee that this makes sense.) The map (\ref{eq:phi z }) is a Hamiltonian $S^1$-action on $\C^n$. (It is generated by the function $H:\C^n\to\R$ defined by $H(q+ip):=\sum_{i=1}^{n_1}q_ip_{n_1+i}-q_{n_1+i}p_i$.)
\begin{claim}\label{claim:X} The set 
\[X:=\bigcup_{z\in S^1}\phi^z(X_0)\]
satisfies the conditions of statement (\ref{thm:exists X:k'}).
\end{claim}
\begin{pf}[Proof of Claim \ref{claim:X}] Since the Stiefel manifolds are compact, the set $X_0$ and hence {\bf $X$ is compact}. To see that {\bf $X$ is contained in $\BAR{B}^{2n}$}, note that $|\Psi(\Theta)|\leq|\Theta|\leq1,$ for every $\Theta=(\Theta_1,\ldots,\Theta_\ell)\in\Cross_iV(k_i,n_i,a)$. It follows that $X_0\sub\BAR{B}^{2n}$. Since $\phi^z$ is orthonormal, for every $z\in S^1$, this implies that $X\sub\BAR{B}^{2n}$. To see that {\bf $X$ has Hausdorff dimension at most $d+1$}, observe that
\begin{equation}\label{eq:dim Cross}\dim\big(\Cross_{i=1}^\ell V(k_i,n_i,a)\big)=\sum_{i=1}^\ell k_i(2n_i-k_i)\leq d,\end{equation}
where in the second step we used the second inequality in (\ref{eq:k i n i d geq}). Note that $X$ is the image of $S^1\x\Cross_{i=1}^\ell V(k_i,n_i,a)$ under the smooth map $(z,\Theta_1,\ldots,\Theta_\ell)\mapsto\phi^z\circ\Psi(\Theta_1,\ldots,\Theta_\ell)$. Combining this with (\ref{eq:dim Cross}), a standard result (cf.~\cite[p.~176]{Fed}) implies that $X$ has Hausdorff dimension at most $d+1$. 

Recalling that $a=\pi/{\kkk(n,d)}$, {\bf inequality (\ref{eq:diam X})} is a consequence of the definition (\ref{eq:diam X Y om}) and the following claim.
\begin{claim}\label{claim:Vert phi a} We have
\begin{equation}\label{eq:Vert phi a}\Vert\phi^{-i}|_X\Vert_X^{\R^{2n},\om_0}\geq a.
\end{equation}
\end{claim}
\begin{pf}[Proof of Claim \ref{claim:Vert phi a}] We define $(M,\om):=(\C^n,\om_0)$, $z_0:=\frac\pi2$ and the map $[0,1]\x \C^n\ni(t,x)\mapsto \psi^t(x)\in\C^n$ by 
\begin{equation}\label{eq:psi t x i}(\psi^t(x))^i:=\left\{\begin{array}{ll}
(1+t)x^i,&\textrm{if }\lf(i)=1,\\
x^i,&\textrm{otherwise.}
\end{array}\right.
\end{equation}
\begin{claim}\label{claim:hyp X} The hypotheses of Lemma \ref{le:d X limsup} are satisfied. 
\end{claim}
\begin{proof}[Proof of Claim \ref{claim:hyp X}] The {\bf condition (\ref{eq:psi 0 id})} clearly holds. We define
\begin{equation}\label{eq:f R 2n}f:\C^n=\C^{n_1}\x\C^{n-n_1}\to\R,\quad f(y,y'):=|y|^2-\frac1{\kkk(n,d)},\end{equation}
That the pair {\bf $(X_0,df)$ is rigidifying}, is a consequence of the following claim.
\begin{claim}\label{claim:pr hyp le:rigid} The hypothesis of Lemma \ref{le:rigid} is satisfied. 
\end{claim}
\begin{proof}[Proof of Claim \ref{claim:pr hyp le:rigid}] We define 
\[(\wt M,\wt\om):=\big(\Cross_{i=1}^\ell\C^{k_i\x n_i},\om_0\big),\,(M',\om'):=(\ker\Psi,\om_0|_{M'}),\,G:=\Cross_{i=1}^\ell\U(k_i),\]
where $\Psi$ is defined as in (\ref{eq:Psi Cross}) and $\U(k)\sub\C^{k\x k}$ denotes the unitary group. Furthermore, we define the action of $G$ on $\wt M$ by 
\begin{equation}\label{eq:U}(U_1,\ldots,U_\ell)\cdot(\Theta_1,\ldots,\Theta_\ell):=\big(U_1\Theta_1,\ldots,U_\ell\Theta_\ell\big).\end{equation}
Moreover, we identify the Lie algebra $\g$ of $G$ with its dual via the inner product given by the trace, and we define the map $\mu:\Cross_{i=1}^\ell\C^{k_i\x n_i}\to\g^*\iso\g$ by 
\[\mu(\Theta_1,\ldots,\Theta_\ell):=\left(\frac i2\big(\frac1{\kkk(n,d)}\one_{k_1}-\Theta_1\Theta_1^*\big),\ldots,\frac i2\big(\frac1{\kkk(n,d)}\one_{k_\ell}-\Theta_\ell\Theta_\ell^*\big)\right).\]
Finally, we define $\Psi':\wt M\to M'\sub\wt M$ to be the projection along ${M'}^{\om_0}$ (the symplectic complement of the linear symplectic subspace $M'$ of $\wt M$), and 
\begin{equation}\label{eq:psi}\wt\Psi:=(\Psi,\Psi'):\wt M\to M\x M'.\end{equation}
We show that the conditions of Lemma \ref{le:rigid} are satisfied with $\psi:=\wt\Psi$: The action (\ref{eq:U}) is Hamiltonian, and $\mu$ is a moment map. We denote by $\pr:M\x M'\to M$ the canonical projection. Condition (\ref{eq:X pr mu}) follows from the facts $\mu^{-1}(0)=\Cross_iV(k_i,n_i,a)$ and $\pr\circ\wt\Psi=\Psi$, and (\ref{eq:X 0 Psi}). 

Furthermore, we define $g:\g\to\R$ by $g(\xi_1,\ldots,\xi_\ell):=2i(\xi_1)^1_1$. Using (\ref{eq:f R 2n},\ref{eq:L K N n1},\ref{eq:Psi Cross}) and the fact $\pr\circ\wt\Psi=\Psi$, it follows that $g\circ\mu=f\circ\pr\circ\wt\Psi$. This proves that $\mu$ factors through $f\circ\pr\circ\wt\Psi$ and completes the proof of Claim \ref{claim:pr hyp le:rigid}.
\end{proof}
\begin{claim}\label{claim:X 0 phi} The {\bf inclusions (\ref{eq:X phi Bar z 0})} hold.
\end{claim}
\begin{proof}[Proof of Claim \ref{claim:X 0 phi}] Let $x:=(y,y')\in X_0$. By (\ref{eq:X 0 Psi}) this means that there exists a tuple $\Theta:=(\Theta_1,\ldots,\Theta_\ell)\in \Cross_iV(k_i,n_i,a)$ such that $\Psi(\Theta)=(y,y')$. To see the first inclusion in (\ref{eq:X phi Bar z 0}), observe that by (\ref{eq:L K N n1},\ref{eq:Psi Cross}), $y$ is the first row of $\Theta_1$. Using (\ref{eq:f R 2n}), it follows that $f(y,y')=0$. This proves the first inclusion.

To see that the second inclusion holds, we denote $y'=:(y'',y''')\in\C^{n_1}\x\C^{n-2n_1}$. Using (\ref{eq:phi z },\ref{eq:f R 2n}), we have 
\begin{equation}\label{eq:f phi y}f\circ\phi^i(y,y'',y''')=f(-y'',y,y''')=|y''|^2-\frac1{\kkk(n,d)}.\end{equation}
It follows from (\ref{eq:L K N n2},\ref{eq:Psi Cross}) that $|y''|$ is bounded above by the norm of the first row of $\Theta_2$, i.e., $1/\sqrt{\kkk(n,d)}$. Combining this with (\ref{eq:f phi y}), we have $f\circ\phi^{-i}(x)\leq0$. The second inclusion in (\ref{eq:X phi Bar z 0}) follows. This proves Claim \ref{claim:X 0 phi}.
\end{proof}
We check the last hypothesis of Lemma \ref{le:d X limsup}:
\begin{claim}\label{claim:dd t f} For every $x\in X_0\sub\C^n$ the {\bf inequalities (\ref{eq:dd t f psi t})} hold.
\end{claim}
\begin{proof}[Proof of Claim \ref{claim:dd t f}] To see that the {\bf first inequality} holds, we denote $x=:(y,y')\in\C^{n_1}\x\C^{n-n_1}$. By (\ref{eq:X 0 Psi}) there exists a tuple $\Theta:=(\Theta_1,\ldots,\Theta_\ell)\in\Cross_iV(k_i,n_i,a)$ such that $\Psi(\Theta)=x$. It follows from (\ref{eq:L K N n1},\ref{eq:psi t x i},\ref{eq:f R 2n}) that 
\begin{equation}\label{eq:dd t f 2}\left.\frac d{dt}\right|_{t=0}(f\circ\psi^t(x))=2|y|^2.\end{equation}
Furthermore, (\ref{eq:L K N n1},\ref{eq:Psi Cross}) imply that $y$ is the first row of $\Theta_1$, and therefore has norm $1/\sqrt{\kkk(n,d)}$. Combining this with (\ref{eq:dd t f 2}), it follows that the first inequality in (\ref{eq:dd t f psi t}) hold.

We show that the {\bf second inequality} holds: Using (\ref{eq:L K N n2},\ref{eq:Psi Cross},\ref{eq:psi t x i}), we have $\big(\phi^i\circ\psi^t(x)\big)^j=x^{n_1+j}$, for every $j\in\{1,\ldots,n_1\}$, and therefore,
\[\left.\frac d{dt}\right|_{t=0}\big(f\circ\phi^i\circ\psi^t(x)\big)=0.\]
Hence the second inequality in (\ref{eq:dd t f psi t}) is satisfied. This proves Claim \ref{claim:dd t f}. \end{proof}
Hence all hypotheses of Lemma \ref{le:d X limsup} are satisfied. This completes the proof of Claim \ref{claim:hyp X}. 
\end{proof}
By Claim \ref{claim:hyp X}, we may apply Lemma \ref{le:d X limsup}, to conclude that inequality (\ref{eq:d X limsup}) holds. 

Let now $t\in[0,1]$. We define 
\[N^t:=V\big(k_1,n_1,(1+t)^2a\big)\x\Cross_{i=2}^\ell V(k_i,n_i,a).\] 
We denote by $\pr:\C^n\x\ker\Psi\to\C^n$ the canonical projection. Recall the definition (\ref{eq:psi}) of $\wt\Psi$. Using (\ref{eq:L K N n1},\ref{eq:Psi Cross},\ref{eq:X 0 Psi},\ref{eq:psi t x i}), we have $\pr\circ\wt\Psi(N^t)=\Psi(N^t)=\psi^t(X_0)$. Hence it follows from Remark \ref{rem:e M M'} below that 
\begin{equation}\label{eq:e psi t X 0}e\big(\psi^t(X_0),\C^n,\om_0\big)\geq e\big(\wt\Psi(N^t),\C^n\x\ker\Psi,\om_0\big).\end{equation}
Furthermore, by an elementary argument, the map $\wt\Psi$ is a (linear) symplectomorphism, and therefore, 
\begin{equation}\label{eq:e psi N t}e\big(\wt\Psi(N^t),\C^n\x\ker\Psi,\om_0\big)=e\big(N^t,\Cross_i\C^{k_i\x n_i},\om_0\big).\end{equation}
Note that $N^t$ is a closed regular coisotropic submanifold of $\Cross_i\C^{k_i\x n_i}$. Hence applying Theorem \ref{thm:N phi N}, it follows that 
\begin{equation}\label{eq:e N t}e\big(N^t,\Cross_i\C^{k_i\x n_i},\om_0\big)\geq A_\Cross\big(\Cross_i\C^{k_i\x n_i},\om_0,N^t\big).\end{equation}
It follows from Remark \ref{rmk:S N N'} and Proposition \ref{minimalareaStiefel} below that
\begin{equation}\label{eq:A Cross C}A_\Cross\big(\Cross_i\C^{k_i\x n_i},\om_0,N^t\big)\geq\min\{(1+t)^2a,a\}=a.\end{equation}
Combining this with (in-)equalities (\ref{eq:d X limsup},\ref{eq:e psi t X 0},\ref{eq:e psi N t},\ref{eq:e N t}), inequality (\ref{eq:Vert phi a}) follows. This proves Claim \ref{claim:Vert phi a} and hence Claim \ref{claim:X}, and completes the proof of statement (\ref{thm:exists X:k'}) and therefore of Theorem \ref{thm:exists X}.
\end{pf}
\end{pf}
\end{proof}
\appendix
\section{Auxiliary results}\label{sec:aux}
\subsection{(Pre-)symplectic geometry}\label{subsec:pre sympl}
The following lemma was used in the proof of Proposition \ref{prop:H wt H}. 
\begin{lem}\label{le:X H r} Let $(M,\om)$ be a symplectic manifold, $N\sub M$ a symplectic submanifold, and $r:M\to N$ a smooth retraction such that $\ker dr(x)=T_xN^\om$ (as defined in (\ref{eq:W om})), for every $x\in N$. Let $H\in C^\infty(N,\R)$. Then we have, for every $x\in N$, 
\begin{equation}\label{eq:X H r} X^\om_{H\circ r}(x)=X^{\om|_N}_H(x).
\end{equation}
\end{lem}
For the proof of Lemma \ref{le:X H r} we need the following. Let $(V,\om)$ be a symplectic vector space and $W\sub V$ a linear subspace. Assume that $W$ is a symplectic subspace. We denote by $\pr^W$ the linear projection from $V$ onto $W$, along $W^\om$. 
\begin{rem}\label{rem:V W} Let $v\in V$ and $w\in W$ be vectors such that $\om(v,\cdot)=\om|_W(w,\pr^W\cdot)$. Then we have $v=w$. This follows from a straight-forward argument. $\Box$
\end{rem}
\begin{proof}[Proof of Lemma \ref{le:X H r}]\setcounter{claim}{0} Let $x\in N$. We have
\[\om\big(X^\om_{H\circ r}(x),\cdot)=d(H\circ r)(x)=dH(x)dr(x)=\om|_N\big(X^{\om|_N}_H(x),dr(x)\cdot\big).\]
Since $r$ is a retraction onto $N$, the map $dr(x):T_xM\to T_xM$ is a projection onto $T_xN$. By hypothesis its kernel is $T_xN^\om$. Hence equality (\ref{eq:X H r}) follows from Remark \ref{rem:V W}. This proves Lemma \ref{le:X H r}.
\end{proof}
We used the following remark in the proof of Theorem \ref{thm:exists X}. 
\begin{rem}\label{rmk:S N N'} Let $(M,\om)$ and $(M',\om')$ be symplectic manifolds, and $N\sub M$ and $N'\sub M'$ coisotropic submanifolds. Then 
\[S\big(M\x M',\om\oplus\om',N\x N'\big)=S(M,\om,N)+S(M',\om',N').\]
This follows from a straight-forward argument. $\Box$
\end{rem}
The next result was used in the proof of Theorem \ref{thm:exists X}. For $k,n\in\N$ satisfying $k\leq n$ we denote
\[V(k,n):=\big\{\Theta\in\C^{k\x n}\,\big|\,\Theta\Theta^*=\one_k\big\}.\]
\begin{prop}\label{minimalareaStiefel}
The Stiefel manifold $V(k,n)$ has minimal area
\[A(\R^{2kn},\om_0,V(k,n))=\pi.\]
\end{prop}
\begin{proof}\setcounter{claim}{0}
For a proof we refer to \cite[Proposition 1.3]{ZiLeafwise}.
\end{proof}
We used the next remark in the proof of Theorem \ref{thm:exists X}(\ref{thm:exists X:k'}). Recall the definition (\ref{eq:e X M om}) of the displacement energy. 
\begin{rem}\label{rem:e M M'} Let $(M,\om)$ and $(M',\om')$ be symplectic manifolds and $X\sub M$ a subset. Then we have
\[e\big(X\x M',M\x M',\om\oplus\om'\big)\leq e(X,M,\om).\]
This follows from a straight-forward argument. $\Box$
\end{rem}
The next lemma was used in the proof of Proposition \ref{prop:Vert}(\ref{prop:Vert:sympl}). For a proof see \cite[Lemma 35]{SZSmall}. We denote by $\Ham_c(M,\om)$ the group of Hamiltonian diffeomorphisms of $M$ generated by a compactly supported function, and by $\Vert\cdot\Vert^{M,\om}_c$ the compactly supported Hofer norm on this group. 
\begin{lem}[\cite{SZSmall}]\label{le:phi psi} Let $(M,\om)$ be a symplectic manifold, $K\sub M$ a compact subset, $\phi\in\Ham(M,\om)$, and $\eps>0$. Then there exists $\psi\in\Ham_c(M,\om)$ such that 
\begin{equation}\label{eq:psi K d c}\psi|_K=\phi|_K,\quad\Vert\psi\Vert^{M,\om}_c\leq\Vert\phi\Vert^{M,\om}+\eps.
\end{equation} 
(Here our convention is that $\infty+\eps:=\infty$.)
\end{lem}
\subsection{Topology and manifolds}\label{subsec:top mf}
In Section \ref{subsec:proofs:prop:Ham X om,Vert} we used the following remark. Let $M$ be a $C^\infty$-manifold and $V$ a time-dependent vector field on $M$, i.e., a smooth map $[0,1]\x M\in(t,x)\mapsto V^t(x)\in TM$ such that $\pi\circ V^t=\id_M$, where $\pi:TM\to M$ denotes the canonical projection. We denote by $\DD_V\sub [0,1]\x M$ the domain of the flow of $V$, and by $\phi_V$ the flow of $V$. 
\begin{rem}\label{rem:phi X t} The set $\DD_V$ is open, and for every $t\in[0,1]$, the map 
\[\phi_V^t:\DD_V^t:=\big\{x\in M\,\big|\,(t,x)\in\DD_V\big\}\to M\] 
is injective and an immersion. (This follows for example from \cite[Theorem 17.15, p.~451, and Problem 17-15, p.~463]{Le}.)
\end{rem}
The following lemma was used in the proof of Proposition \ref{prop:H wt H}.
\begin{lem}\label{le:Urysohn} Let $M$ be a smooth manifold and $A_i\sub M$ a closed subset, for $i=0,1$. If $A_0\cap A_1=\emptyset$ then there exists a function $f\in C^\infty(M,[0,1])$ such that $f|_{A_i}\const i$, for $i=0,1$. 
\end{lem}
\begin{proof}[Proof of Lemma \ref{le:Urysohn}]\setcounter{claim}{0} This follows from a $C^\infty$-version of Urysohn's Lemma for $\R^n$ (see for example Theorem 1.1.3, p.~4 in \cite{KP}) and a partition of unit argument.
\end{proof}
For the proof of Proposition \ref{prop:H wt H} we need the following result. Let $M$ be a smooth manifold, $N\sub M$ a submanifold, and $E\sub TM|_N$ a subbundle such that $TM|_N$ is the direct sum of $TN$ and $E$. For $x\in N$ we denote by $E_x$ the fiber of $E$ over $x$.
\begin{prop}\label{prop:tubular} Assume that $N$ is closed as a subset of $M$. Then there exists an embedding $\psi:E\to M$ such that, identifying $N$ with the zero section of $E$, we have
\begin{eqnarray}\label{eq:psi N}&\psi|_N=\id_N,&\\
\label{eq:d psi x v}&d\psi(x)v=v,\quad\forall v\in T_x(E_x)=E_x,\,x\in N.&
\end{eqnarray}
\end{prop}
\begin{proof}[Proof of Proposition \ref{prop:tubular}]\setcounter{claim}{0} This follows from a standard argument, along the lines of the proof of Theorem 5.2 in the book \cite{Hi}.
\end{proof}
The following lemma was used in the proof of Lemma \ref{le:d X limsup}.
\begin{lem}\label{le:X 0 phi psi} Let $M$ be a $C^1$-manifold, $X_0\sub M$ a compact subset, $\phi\in C^1(M,M)$ and $[0,1]\x M\ni(t,x)\mapsto\psi^t(x)\in M$ a $C^1$-map. Assume that $\psi^0=\id$ and there exists a function $f\in C^\infty(M,\R)$ such that 
\begin{eqnarray}\label{eq:f X 0}&f(X_0)\sub(-\infty,0],\quad f\circ\phi(X_0)\sub[0,\infty),&\\
\label{eq:f psi t x}&\left.\frac d{dt}\right|_{t=0}(f\circ\psi^t)(x)\leq0,\quad\left.\frac d{dt}\right|_{t=0}(f\circ\phi\circ\psi^t)(x)>0,\,\forall x\in X_0.&
\end{eqnarray}
Then there exists $t_0>0$ such that for every $t\in(0,t_0]$, we have
\[\phi\circ\psi^t(X_0)\cap\psi^t(X_0)=\emptyset.\]
\end{lem}
In the proof of this result, we will use the following remark:
\begin{rem}\label{rmk:f t x} Let $X$ be a compact topological space and $[0,1]\x X\ni(t,x)\mapsto f^t(x)\in\R$ a function. Assume that the partial derivative $\dd_tf:[0,1]\x X\to \R$ exists and is continuous. Then we have 
\[\lim_{t\searrow0}\frac1t\max_{x\in X}\left|f^t(x)-f^0(x)-t\dd_t|_{t=0}f^t(x)\right|=0.\]
This follows from an elementary argument. 
\end{rem}
\begin{proof}[Proof of Lemma \ref{le:X 0 phi psi}]\setcounter{claim}{0} Using that $\psi^0=\id$ and that $X_0$ is compact, it follows from Remark \ref{rmk:f t x} that 
\begin{equation}\label{eq:f psi t t dd t}\lim_{t\searrow0}\frac1t\max_{x\in X_0}\left|f\circ\psi^t(x)-f(x)-t\left.\frac d{dt}\right|_{t=0}(f\circ\psi^t(x))\right|=0.
\end{equation}
Furthermore, defining the function $g:[0,1]\x M\to\R$ by
\[g(t,x):=f\circ\phi\circ\psi^t(x)-f\circ\phi(x)-t\left.\frac d{dt}\right|_{t=0}\big(f\circ\phi\circ\psi^t(x)\big),\]
the same remark implies that 
\begin{equation}\label{eq:f phi psi t dd t}\lim_{t\searrow0}\frac1t\max_{x\in X_0}|g(t,x)|=0.
\end{equation}
We denote
\[c:=\min_{x,y\in X_0}\left(\left.\frac d{dt}\right|_{t=0}(f\circ\phi\circ\psi^t)(y)-\left.\frac d{dt}\right|_{t=0}(f\circ\psi^t)(x)\right),\]
\[d(t):=\frac1t\min_{x,y\in X_0}\big(f\circ\phi\circ\psi^t(y)-f\circ\psi^t(x)\big).\]
It follows from (\ref{eq:f X 0},\ref{eq:f psi t t dd t},\ref{eq:f phi psi t dd t}) that
\begin{equation}\nn\liminf_{t\searrow0}d(t)\geq c.\end{equation}
Hence there exists $t_0\in(0,1]$ such that for $t\in(0,t_0]$ we have $d(t)\geq\frac c2$. Compactness of $X_0$ and (\ref{eq:f psi t x}) imply that $c>0$. It follows that for every $t\in(0,t_0]$ and $x,y\in X_0$, we have $f\circ\psi^t(x)\neq f\circ\phi\circ\psi^t(y)$, and therefore, $\psi^t(x)\neq \phi\circ\psi^t(y)$. It follows that $\psi^t(X_0)\cap\phi\circ\psi^t(X_0)=\emptyset$, for every $t\in(0,t_0]$. This proves Lemma \ref{le:X 0 phi psi}.
\end{proof}
The next lemma implies that the Hofer semi-norm given by (\ref{eq:Vert H X}) is well-defined. 
\begin{lem}\label{le:f} Let $X$ be a topological space and $f:[0,1]\x X\to \R$ be a continuous function. Assume that there exists a sequence of compact subsets $K_\nu\sub X$, $\nu\in\N$ such that $\bigcup_\nu K_\nu=X$. Then the map 
\[[0,1]\ni t\mapsto \sup_{x\in X}f(t,x)\]
is Borel measurable. 
\end{lem}
\begin{proof} \setcounter{claim}{0}This follows from an elementary argument. 
\end{proof}


\begin{thebibliography}{99}

\bibitem[Ch]{Ch} Yu.~Chekanov, \emph{Lagrangian intersections, symplectic energy, and areas of holomorphic curves}, Duke Math. J. {\bf 95} (1998), no. {\bf 1}, 213--226.

\bibitem[Fed]{Fed} H.~Federer, \emph{Geometric measure theory}. Die Grundlehren der mathematischen Wissenschaften, Bd.~{\bf 153}, Springer-Verlag, New York 1969.

\bibitem[Hi]{Hi} M. W. Hirsch, \emph{Differential Topology}, Graduate Texts in Mathematics, no.~{\bf 33}, Springer-Verlag, New York-Heidelberg, 1976. 

\bibitem[HZ]{HZ} H.~Hofer and E.~Zehnder, \emph{Symplectic Invariants and Hamiltonian Dynamics}, Birkh\"auser Advanced Texts, Birkh\"auser Verlag, Basel, Boston, Berlin, 1994. 

\bibitem[KP]{KP} St. G. Krantz, H. R. Parks, \emph{The Geometry of domains in space}, Birkh\"auser Advanced Texts: Basler Lehrb\"ucher, Birkh\"auser, 1999. 

\bibitem[LM]{LMGeo} F.~Lalonde and D.~McDuff, \emph{The geometry of symplectic energy}, Ann.~Math.~{\bf 141} (1995), no.~{\bf 2}, 349--371. 

\bibitem[Le]{Le} J.~M.~Lee, \emph{Introduction to smooth manifolds}, Graduate Texts in Mathematics, {\bf 218}, Springer-Verlag, New York, 2003. 

\bibitem[McD]{McD} D.~McDuff, \emph{Loops in the Hamiltonian group: a survey}, arxiv:0711.4086v2 .

\bibitem[Os]{Os} Y.~Ostrover, \emph{A comparison of Hofer's metrics on Hamiltonian diffeomorphisms and Lagrangian submanifolds}, Commun.~Contemp.~Math.~{\bf 5} (2003), {\bf no.~5}, 803--811. 

\bibitem[Schl]{Schl} F. Schlenk, \emph{Embedding problems in symplectic geometry}, de Gruyter Expositions in Mathematics 40, W. de Gruyter, 2005. 

\bibitem[Si]{Si} J.-C.~Sikorav, \emph{Syst\`{e}mes Hamiltoniens et topologie symplectique}, Dipartimento di Mathematica dell' Universit\`{a} di Pisa, ETS, EDITRICE PISA, 1990.

\bibitem[SZ]{SZSmall}J.~Swoboda and F.~Ziltener, \emph{Coisotropic Displacement and Small Subsets of a Symplectic Manifold}, arXiv:1101.0920.

\bibitem[Zi1]{ZiLeafwise} F. Ziltener, \emph{Coisotropic Submanifolds, Leaf-wise Fixed Points, and Presymplectic Embeddings}, J.~Symplectic Geom.~{\bf 8} (2010), no.~{\bf 1}, 1--24.

\bibitem[Zi2]{ZiHofer} F. Ziltener, \emph{Relative Hofer Geometry and the Asymptotic Hofer-Lipschitz Constant}, in preparation. %%% arXiv

\end{thebibliography}
\end{document}